\documentclass[11pt,reqno]{amsart}
\usepackage{amsmath, amssymb, graphicx, mathrsfs, hyperref}
\usepackage{amsthm}
\usepackage{palatino}
\usepackage{autonum}
\usepackage[makeroom]{cancel}
\usepackage{enumerate} 
\usepackage[usenames,dvipsnames,svgnames]{xcolor}
\usepackage[top=1in, bottom=1in, left=1in, right=1in]{geometry} 
\allowdisplaybreaks

\renewcommand{\Im}{\operatorname{Im}}

\renewcommand{\Re}{\operatorname{Re}}
\renewcommand{\Im}{\operatorname{Im}}
\newcommand{\s}{{\sigma}}

\renewcommand{\a}{\alpha}
\renewcommand{\b}{\beta}
\newcommand{\e}{\epsilon}
\renewcommand{\d}{{\delta}}
\newcommand{\g}{\gamma}
\newcommand{\G}{\Gamma}

\renewcommand{\r}{{\rho}}

\newcommand{\z}{\zeta}

\newcommand{\bs}{\boldsymbol}

\renewcommand{\(}{\left\(}
\renewcommand{\)}{\right\)}
\renewcommand{\[}{\left\[}

\renewcommand{\]}{\right\]}

\numberwithin{equation}{section}
\theoremstyle{plain}
\newtheorem{theorem}{Theorem}[section]
\newtheorem{lemma}[theorem]{Lemma}
\newtheorem{remark}[]{Remark}
\newtheorem*{remark*}{Remark}

\newtheorem{corollary}[theorem]{Corollary}

\newtheorem{example}[]{Example}

\theoremstyle{definition}
\newtheorem{definition}{Definition}[section]

\makeatletter
\def\proof{\@ifnextchar[{\@oproof}{\@nproof}}
\def\@oproof[#1][#2]{\trivlist\item[\hskip\labelsep\textit{#2 Proof of\
		#1.}~]\ignorespaces}
\def\@nproof{\trivlist\item[\hskip\labelsep\textit{Proof.}~]\ignorespaces}

\makeatother

\begin{document}
	\title[Equivalence between  Functional Equation and Vorono\"{\i}-type  summation identities]{Equivalence between the Functional Equation and Vorono\"{\i}-type  summation identities for a class of $L$-Functions} 
	
	\author{Arindam Roy, Jagannath Sahoo and Akshaa Vatwani}
	\thanks{2020 \textit{Mathematics Subject Classification.} Primary 11M06, Secondary 11M26, 11M41.\\
		\textit{Keywords and phrases.}  Hecke functional equation, Vorono\"{\i} summation formula,  modular relations.}
	\address{Department of Mathematics, University of North Carolina at Charlotte, NC 28223, USA}
		\email{arindam.roy@charlotte.edu}
	\address{Department of Mathematics, Indian Institute of Technology Gandhinagar, Palaj, Gandhinagar 382355, Gujarat, India} 
	\email{jagannath.sahoo@iitgn.ac.in, akshaa.vatwani@iitgn.ac.in}
	\begin{abstract}
To date, the best methods  for estimating the growth  of mean values of arithmetic functions rely on   the Vorono\"{\i} summation formula. By noticing a general pattern in the proof of his summation formula, Vorono\"{\i} postulated that  analogous summation formulas  for $\sum a(n)f(n)$ can be obtained with  ``nice" test functions $f(n)$,  provided $a(n)$ is an ``arithmetic function". These arithmetic functions $a(n)$ are called so because they are  expected to appear as  coefficients of some $L$-functions satisfying certain properties.  
It has been well-known that the functional	equation for a general $L$-function can be used to derive a Vorono\"{\i}-type  summation identity for that $L$-function. In this article, we show  that  such a  Vorono\"{\i}-type  summation identity  in fact endows the $L$-function with some  structural properties, yielding in particular  the functional equation.  
We do this by considering  Dirichlet series satisfying functional equations involving multiple Gamma factors and show that a given arithmetic function appears as a coefficient of such a Dirichlet series  \textit{if and only if} it satisfies the aforementioned summation formulas. 
	\end{abstract}
	\maketitle
	\section{Introduction}
	\label{sec: results}

In number theory and many other branches of mathematics, distributions of finite sums of arithmetic functions  constitute an active area of interest. An important tool for calculating the finite sum of a given  arithmetic function is a summation formula of the form 
\begin{align}
		\label{eqn:summationformula}
	\sideset{}{'}\sum_{a\leq n\leq b}f(n)=\int_{a}^{b}f(x)\,dx+2\sum_{n=1}^{\infty}\int_{a}^bf(x)\cos(2\pi n x)\,dx.
\end{align}
Here the function $f$ can be taken as a function of bounded variation that is supported in a finite compact interval and the dash on the sum indicates  that the last term of the sum is multiplied by  $\frac{1}{2}$ if $n=b$. The Gauss circle problem and the Dirichlet divisor problem are two well-known problems in this direction. Dirichlet showed that if $d(n)$ denotes the number of positive divisors of an integer $n$, then
\begin{align}
	\Delta_d(x):=\sum_{n\leq x}d(n)-x\log x-(2\gamma -1)x=O(x^{1/2}).
\end{align}
On the other hand, if $r_2(n)$ is the number of ways to represent a positive integer $n$ as a sum of squares of two integers then Gauss showed that 
\begin{align}
	\Delta_r(x)=\bigg\lvert\sum_{n\leq x}r_2(n)-\pi x\bigg\rvert=O(x^{1/2}).
\end{align} 
Much later, in  1904, Vorono\"{\i} \cite{VoronoiI,VoronoiII}  generalized the summation formula \eqref{eqn:summationformula} to improve the bounds for $\Delta_d(x)$ and $\Delta_r(x)$ to $O(x^{1/3})$. In particular, he obtained in \cite{VoronoiI} 
\begin{align}
	\sideset{}{'}\sum_{a\leq n \leq b}d(n)f(n)=\int_a^bf(x)(\log x+2\gamma)\,dx+\sum_{n=1}^{\infty}d(n)\int_a^bf(x)L_0(4\pi\sqrt{nx})\,dx\label{eq:vsd}
\end{align}
and 
\begin{align}
	\sideset{}{'}\sum_{a\leq n \leq b}r_2(n)f(n)=\pi\sum_{n=1}^{\infty}r_2(n)\int_a^bf(x)J_0(2\pi\sqrt{nx})\,dx,\label{eq:vsr}
\end{align}
where $f$ is a piece-wise monotonic and continuous function, $L_0(x):=4K_0(x)-2\pi Y_0(x)$,  $K_0$ is the modified Bessel function of the second kind, and $J_0, Y_0$ are Bessel functions of the first and second kind respectively. 			

Vorono\"{\i} summation is the best known method for improving the bounds for $\Delta_d(x)$ and $\Delta_r(x$) and was crucial in understanding the lower bounds of these two error terms. By using  \eqref{eq:vsd} and \eqref{eq:vsr}, Hardy and Landau observed in \cite{HardyLattice} that $\Delta_d(x)$ and $\Delta_r(x)$ cannot be $O(x^{1/4})$.  Consequently, they conjectured that both $\Delta_d(x)$ and $\Delta_r(x)$ are $O_\e(x^{1/4+\epsilon})$ for any $\e>0$.

Returning to the work in \cite{VoronoiI, VoronoiII}, Vorono\"{\i} found a common pattern in the proofs of \eqref{eq:vsd} and \eqref{eq:vsr}. This led  him to propose  analogous summation formulas of the type \eqref{eq:vsd} and \eqref{eq:vsr} for $\sum_{a\leq n\leq b}a(n)f(n)$, where $a(n)$ is any ``arithmetic function" and $f(n)$ is a ``nice" test function. Without evidence, it is generally accepted that $a(n)$ is chosen to be the coefficient of some known $L$-function in order to produce analogues of the Vorono\"{\i} summation formula.  The choices for $a(n)$ now also include automorphic forms as seen in \cite{GoldfeldLiI}, \cite{MilschmidI}, etc.  The selection of the test function $f(n)$ is an important component of the Vorono\"{\i} summation formula. Ideally, one would like to choose  discontinuous  characteristic functions as candidates for  $f(n)$, since one is primarily interested in the partial sum $\sum_{a\leq n\leq b}a(n)$. However, one notes that  smooth functions serve as good approximations to these test functions and due to the rapid decay of smooth functions, the integral transforms on the right-hand side of the summation formula (analogous to those on the right-hand side of \eqref{eq:vsd} or  \eqref{eq:vsr}) can be analyzed well. On the other hand, due to the approximation error, a small amount of information is lost when a characteristic function is replaced  by a smooth function. 

The main objective of this article is to revisit the coefficients $a(n)$ in the Vorono\"{\i} summation formula. 
We consider a broad class  $\mathcal{C}$ of Dirichlet series satisfying functional equations of Hecke-type and  derive  analogues of the Vorono\"{\i} summation formula for coefficients of these series. While results in this direction have been obtained  in a lot of generality, the key aspect of this article is that we demonstrate the converse. More precisely, we show that if an `arithmetic function' $a(n)$ satisfies such a Vorono\"{\i}-type summation formula, then it must appear as a coefficient of some Dirichlet series within this general class. In particular, we prove an equivalence between the functional equation of Hecke-type for the Dirichlet series corresponding to $a(n)$ and the analogue of the Vorono\"{\i}  summation formula for $a(n)$. While there are multiple results  in  various  settings for the  forward implication  (that is, showing that the functional equation implies a summation formula),     the converse is not known in much generality.   We elaborate  more upon this  in Section  \ref{sec: results}. The converse  is usually proved via an identity known as a modular relation (see for instance \eqref{modrelsing}).  
A unique feature of our result is that   in order to  prove this converse for our general $L$-function, we give an auxiliary modular relation which is different from the one occurring naturally for our Dirichlet series, and has not appeared in the literature previously. Moreover, we show that  this new modular relation is  equivalent to the functional equation.

Modular identities have a rich history, going back to the ubiquitous Poisson summation formula, which serves as the basis for them. The Poisson summation formula is expressed as 
\begin{align}
	\sum_{n\in\mathbb{Z}}f(n)=\sum_{n\in\mathbb{Z}}\hat{f}(n)
\end{align}
for all Schwartz functions $f$. Here $\hat{f}(t):=\int_{\mathbb{R}}f(x)e^{-2\pi i t x}\,dx$ is the Fourier transform of $f$.  If one considers the function $f(x)=e^{-\pi x^2 t}$, $t>0$, then we have the well-known modular identity, also called  the $\theta$-identity:
\begin{align}
	\theta(t):=\sum_{n\in\mathbb{Z}}e^{-\pi n^2 t}=\frac{1}{\sqrt{t}}\sum_{n\in\mathbb{Z}}e^{-\frac{\pi n^2}{t}},   \quad  (t>0).
\end{align}
It was based on this identity that Riemann \cite{Riemann} in his landmark memoir built the theory of the zeta function  and  derived the functional equation of $\zeta(s)$.  
Results obtaining the  $\theta$-identity from the functional equation  were given by Cahen \cite{Cahen},  Hamburger \cite{HamburgerI, HamburgerII,HamburgerIII,HamburgerIV}, Hurwitz (unpublished note, cf. \cite{Oswald}).  

In order to  state our results, we recall the Riesz-mean  sum:  $\sideset{}{'}\sum_{n\leq x}a(n)(x-n)^{\rho}$, with $\rho\ge 0$, where {the  dash on the sum indicates that the last term of the sum has to be multiplied by $\frac{1}{2}$ if $\r=0$ and $x=n$.}  
After the work of Ramanujan on $\tau(n)$ (defined by $\sum_{n=1}^{\infty}\tau(n)z^n=z\prod_{n=1}^{\infty}(1-z^n)^{24}$ for $|z|<1$), Wilton \cite{Wilton} considered the Riesz sum $\sideset{}{'}\sum_{n\leq x}\tau(n)(x-n)^{\rho}$ and obtained the summation formula 
\begin{align}
A_{\rho}(x):=\frac{1}{\Gamma(\rho +1)}\sideset{}{'}\sum_{n\leq x}\tau(n)(x-n)^{\rho}=\left(\frac{1}{2\pi}\right)^{\rho}\sum_{n=1}^{\infty}\left(\frac{x}{n}\right)^{6+\rho/2}\tau(n)J_{12+\rho}(4\pi\sqrt{nx})\label{eq:risu}
\end{align}
for $\rho>0$. Hardy \cite{Hardy} improved this result and proved that $\rho$ can be taken to be $0$. This is the most desirable case. Recall that  for any suitable smooth test function $f$, partial summation gives
\begin{align}
	\sideset{}{'}\sum_{n\leq x}\tau(n)f(n)=A_0(x)f(x)-\int_1^xA_0(t)f'(t)\,dt.\label{partsum}
\end{align}
Substituting the infinite series representation of $A_0(t)$ from \eqref{eq:risu}  into this and performing some  elementary computations yields  analogues of the Vorono\"{\i} summation formula. 
Although the right-hand side of \eqref{eq:risu} converges absolutely and uniformly in any compact interval for large values of $\rho$,  it can be seen to be convergent by some other means for  small values of $\rho$.  If we denote $S(x)=\sideset{}{'}\sum_{n\leq x} a(n)$ and $S_1(x)=\sideset{}{'}\sum_{n\leq x} a(n)(x-n)$ for a general arithmetic function $a(n)$, then by partial summation, one has $S_1(x)=\int_1^x S(t)\, dt$. Since  the infinite series representation of $S_1(x)$ (analogous to the right-hand side of \ref{eq:risu}) converges rapidly and analyzing the series is relatively simple, bounding $S(x)$ in terms of $S_1(x)$ is fruitful. This is a key reason for studying Reisz sums. Moreover, if $S(x)$ is monotonic then one can bound $S(x)$ by the difference $S_1(x)-S_1(x(1\pm \delta))$ for some suitable choice of $\delta$.

We now state our main results below.
We will consider a  class $\mathcal{C}$ of Dirichlet series with a more general  functional equation with multiple Gamma factors. This class is  defined as follows. 
\begin{definition}\label{def:fn eqn of F} We say that $\phi(s), \psi(s) \in \mathcal C$ if 
		\begin{enumerate}

			\item $	\phi(s)$ and $ \psi(s)$ are as defined in \eqref{def:dirser} with finite abscissae of absolute convergence $\sigma_a$ and $\sigma_b$ respectively. 
			
			\item $\phi(s)$ and $\psi(s)$ satisfy a functional equation  of the type 
			\begin{equation}\label{generalfneqn}
				Q^sF(s)=\omega Q^{\d-s} \overline{G(\d-\bar{s})},
			\end{equation}
			for some $\delta>0$, 
			with 
			$F(s)= \phi(s)\prod_{i=1}^r \Gamma\left(\alpha_i s+\beta_i \right)$,  $G(s) = \psi(s)\prod_{i=1}^r \Gamma\left(\alpha_i s+\beta_i \right)$,  $r\in \mathbb{N}$, $ \alpha_i>0$,    $\beta_i \in \mathbb C$  with $\Re (\beta_i)\geq 0$, $ Q>0$ and {$\omega \in \mathbb C$} with $|\omega|=1$. More precisely, this means that 
			there exists a domain $D$ which is the exterior of a  compact set $S$, such that in $D$, there is  an analytic function $\chi(s)$ with the properties: 
			\begin{enumerate}[(i)]
				\item $\chi(s)=	Q^sF(s)$ for $\Re(s)> \sigma_a$ and $\chi(s)=	\omega Q^{\d-s} \overline{G(\d-\bar{s})}$ for $\Re(s)< \delta-\sigma_b$. 
				\item For some constant $\theta<1$, 
				\begin{equation} \label{chigrowthcond}
					\chi(s)=O\left(\exp\left[\exp\left(\frac{\theta \pi|s|}{\s_2-\s_1}\right)\right]\right),
				\end{equation} 
				uniformly in every strip $-\infty<\sigma_1\leq \Re(s) \leq \sigma_2<\infty$. 				
			\end{enumerate}

		\end{enumerate}
	\end{definition}	
	Given the functional equation \eqref{generalfneqn}, we define the vectors $\bs \a,\, \bs \b$ and the quantity $d_F$ by $\bs \a =(\a_1, \cdots, \a_r),\, \bs \b=(\b_1, \cdots, \b_r)$, and $d_F=2\sum_{i=1}^{r}\a_i$. We will also denote the vector $(\bar{\b}_1, \cdots, \bar{\b}_r)$ by $\bs{\bar{\b}}$.
	In the case of  functions in the Selberg class, $d_F$ is said to be the degree of the function, which acts as an invariant.  We will  denote $\frac{d_F}{2}$ by $d_F^{'}$ throughout this paper.
	In \cite{ChandrasekharanV}, Chandrasekharan and Narasimhan considered the same class of Dirichlet series with the stronger restriction: 
	\begin{align} \label{chigrowthCN}
		\lim_{|t|\to\infty}\chi(\sigma+it)=0,
	\end{align}
	uniformly in every bounded strip $-\infty<\sigma_1\leq\sigma\leq\sigma_2<\infty$. In \cite{ChandrasekharanV}, they derived an identity of the type \eqref{riesz sum} and showed that if $\Re b(n)\neq 0$  for at least one value of $n$, then 
$
		\Re[S_{\rho}(x)-Q_{\rho}(x)]=
		\Omega_{\pm}\left(x^{\theta}\right),
$
	where $S_\rho$ and $Q_\rho$ are  as  in \eqref{riiden}, \eqref{def:Q rho} respectively  and  $\theta=\frac{1}{d_F}\left(d_F^{'}\delta+\rho(d_F-1)-1/2\right)$.

	The following two theorems are concerned with the equivalence of the functional equation and the modular relation for our class of Dirichlet series $\mathcal C$. Such equivalences have been  proved  by Bochner  \cite{bochner} and Kanemitsu, Tanigawa, Tsukada \cite{KanemitsuTanigawaTsukada}.   Results proving the modular relation from the functional equation have also been given  by Tsukada \cite{tsukada},  Kanemitsu, Tanigawa, Tsukada \cite{kanemitsu}. From each of these works, one can   derive the  corresponding result for a subset of functions in the class $\mathcal{C}$, namely those for which $\chi$ from Definition \ref{def:fn eqn of F} satisfies  the stricter  restriction \eqref{chigrowthCN}. 
	In  Theorems \ref{thm:FE implies mod rel1}, \ref{thm:mod rel1 implies FE} below, we prove the equivalence of the functional equation and the modular relation for the entire class $\mathcal C$. In particular, we consider the more relaxed growth condition  \eqref{chigrowthcond} on $\chi$.   
	\begin{theorem}
		\label{thm:FE implies mod rel1}
		Let $\phi(s)$ and $\psi(s)$ be two Dirichlet series in the class $\mathcal{C}$.   Then the functional equation 
		\begin{equation}\label{fn eqn}
			Q^sF(s)=\omega Q^{\d-s} \overline{G(\d-\bar{s})},
		\end{equation} where $F(s)= \phi(s)\prod_{i=1}^r \Gamma\left(\alpha_i s+\beta_i \right)$,  $G(s) = \psi(s)\prod_{i=1}^r \Gamma\left(\alpha_i s+\beta_i \right)$, implies the modular relation 
		\begin{equation}\label{mod relation1} 
			\sum_{n=1}^{\infty}a_nZ_{\bs\alpha, 		  \bs \beta}(\lambda_nx)=P(x)+\frac{\omega}{(xQ)^\delta}\sum_{n=1}^{\infty}\bar{b}_n Z_{\bs \alpha,\bs{\bar{\beta}}}\bigg(\frac{\mu_n}{Q^2x}\bigg), \hspace{1cm} (x>0),
		\end{equation}
		where  for $x,a>0$, 
		\begin{equation}
			Z_{\bs\alpha,\bs\beta}(x):=\frac{1}{2 \pi i}\int_{a-i\infty}^{a+i\infty} \prod_{i=1}^{r} \Gamma(\alpha_i s+ \beta_i) x^{-s}ds, \label{def:Z_alpha,beta}
		\end{equation}
		$P(x)$ is a residual function given by
		\begin{align} 
			P(x)= \frac{1}{2 \pi i}\int_{C} F(s) x^{-s}ds, 
			\label{resfnP(x)}
		\end{align}	
		and 	 $C$ denotes a circle of finite radius, lying  inside the domain $D$ and  containing all the singularities of $F(s)$. 
	\end{theorem}
	\begin{theorem} \label{thm:mod rel1 implies FE}
		The modular relation \eqref{mod relation1} implies the functional equation \eqref{fn eqn}.
	\end{theorem}

	In the next result, we show that the functional equation implies a Riesz-sum identity for series in  our class $\mathcal{C}$, with the smallest possible value of $\rho$, analogous to  the $\rho=0$ case in \eqref{eq:risu}. Such a result was proved by Chandrasekharan and Narasimhan  \cite[eq. (1.3)]{ChandrasekharanV}  for Dirichlet series satisfying functional equations with multiple Gamma factors, with the  growth condition \eqref{chigrowthCN} on  $\chi(s)$. Our result below holds for a larger class of functions, since  $\chi$ fulfills the more relaxed growth condition \eqref{chigrowthcond}. 
  
	\begin{theorem} \label{thm:FE implies Rieszsum}
		Let $\phi(s)$ and $\psi(s)$ be two Dirichlet series satisfying Definition \ref{def:fn eqn of F}. Let $a> \max\{0,\s_a,\s_b\}$ be sufficiently large so that all the singularities of $F(s)= \phi(s)\prod_{i=1}^r \Gamma\left(\alpha_i s+\beta_i \right)$ are contained in the strip $\d-a<\Re(s)<a$. Then the functional equation  $$Q^sF (s)=\omega Q^{\delta-s}\overline{G(\delta-\bar{s})},$$ where  $G(s) = \psi(s)\prod_{i=1}^r \Gamma\left(\alpha_i s+\beta_i \right)$, implies the identity
		\begin{multline}\label{riesz sum}
			\frac{1}{\Gamma(\rho+1)}\sideset{}{'}\sum_{\lambda_n \leq x} a_{n}(x-\lambda_n)^\rho\\
			=Q_\rho(x)+\bigg(\frac{\omega x^{\delta+\rho}}{Q^\delta}\bigg) \sum_{n=1}^{\infty}\bar{b}_n\frac{1}{2 \pi i}\int_{a-i\infty}^{a+i\infty} \frac{\Gamma(\delta-s) \prod_{i=1}^{r}\Gamma(\alpha_i s+\bar{\beta}_i)}{\Gamma(1+\delta-s+\rho)\prod_{i=1}^{r}\Gamma(\alpha_i(\delta-s)+\beta_i)}\bigg(\frac{\mu_nx}{Q^2}\bigg)^{-s}ds	, 
		\end{multline}
		for $x>0$, and $\r> (2\s_b-\d)d_F^{'}-\frac{1}{2}$. Here,  $d_F^{'}=\sum_{i=1}^{r}\a_i$ and $Q_\rho(x)$ is a residual function given by
		\begin{equation}
			\label{resfnQ_rho}
			Q_\rho(x)=\frac{1}{2 \pi i}\int_{C}\frac{\phi(s)\Gamma(s)x^{s+\rho}}{\Gamma(s+\rho+1)} ds,
		\end{equation} 
		where $C=C_a$ denotes a circle of finite radius, lying inside the strip $\d-a<\Re(s)<a$,  containing all the singularities of $F(s)$, such that all the singularities of $\phi(s)\Gamma(s)x^{s+\rho}$ which lie in this strip are contained inside $C$.
		
		
	\end{theorem}
	
	\begin{remark}
		The residual function $P(x)$ given in \eqref{resfnP(x)} is the sum of  residues of $F(s) x^{-s}$. This function $P(x)$ is independent of the choice of $a$ as long as $C$  contains all the singularities of $F(s)$.
		On the other hand, the residual function $Q_\rho(x)$ given in \eqref{resfnQ_rho}	is not invariant of the choice of $a$ as $C$ may not encircle all the poles of the integrand. Since the Gamma factor in the numerator of the integrand has infinitely many poles, the number of residues contributed by this factor may vary depending upon $a$. We also remark that in all our results, $C$ can be taken to be a curve with bounded interior rather than a circle with finite radius.
	\end{remark}
	
	\begin{remark}
		Although \eqref{riesz sum} does not appear to be a closed form, one actually has an asymptotic formula for the integral on the right-hand side, given by Lemma \ref{lem:asym for I_rho}. This allows one to express the right-hand side in terms of cosine functions. For practical purposes of computation, this is as convenient as having a Vorono\"{\i}
		summation formula in terms of Bessel functions.
	\end{remark}
	
	The main contribution of this paper is to show that the Riesz sum identity \eqref{riesz sum} implies the functional equation. Our key idea is to prove this by  introducing an auxiliary modular-type relation, different from the known modular relation \eqref{mod relation1}.  
	This new modular-type relation is given by 
	\begin{equation}\label{mod relation2} \sum_{n=1}^{\infty}a_nY_{\bs \alpha,\bs\beta}(\lambda_nx)=P_1(x)+\frac{\omega}{(xQ)^\d}\sum_{n=1}^{\infty}\bar{b}_n X_{\bs\alpha,\bs{\bar{\beta}}}\bigg(\frac{\mu_n}{Q^2x}\bigg), \hspace{1cm} x>0,
	\end{equation}
	where $P_1(x)$ is a residual function which will be defined precisely later,  
	\begin{equation}
		Y_{\bs \alpha,\bs\beta}(x):=\frac{1}{2 \pi i}\int_{a-i\infty}^{a+i\infty} \G(s)\prod_{i=1}^{r} \Gamma(\alpha_i s+ \beta_i) x^{-s}ds,
		\qquad (x,a>0)  \label{def:Y_alpha,beta}
	\end{equation}
	and
	\begin{equation}
		X_{\bs\alpha,\bs{\bar{\beta}}}(x):=\frac{1}{2 \pi i}\int_{a-i\infty}^{a+i\infty} \G(\delta-s)\prod_{i=1}^{r} \Gamma(\alpha_i s+ \bar{\beta_i}) x^{-s}ds,
		\qquad (x,a>0).	\label{def:X_alpha,beta}
	\end{equation}
	
	It is this auxiliary modular-type relation which acts as an intermediate, allowing us to move from the Riesz sum identity \eqref{riesz sum} to the functional equation   \eqref{fn eqn}.  In the following theorem we first derive this new modular-type relation from the Riesz sum identity.  Subsequently,  we will  deduce the functional equation \eqref{fn eqn} from this using similar techniques as in the proof of Theorem \ref{thm:mod rel1 implies FE}.  
	
	\begin{theorem} \label{thm:riesz sum implies modreln 2}
		Let $a> \max\{0,\s_a,\s_b\}$ be sufficiently large so that all the singularities of $F(s)= \phi(s)\prod_{i=1}^r \Gamma\left(\alpha_i s+\beta_i \right)$ are contained in the strip $\d-a<\Re(s)<a$. Let $C=C_a$ be a curve as in the statement of Theorem \ref{thm:FE implies Rieszsum}.  Then, the Riesz sum identity \eqref{riesz sum} implies the modular relation 
		\begin{equation} \label{aux mod relation}
			\sum_{n=1}^{\infty}a_nY_{\bs \alpha,\bs\beta}(\lambda_nx)=P_1(x)+\frac{\omega}{(xQ)^\d}\sum_{n=1}^{\infty}\bar{b}_n X_{\bs\alpha,\bs{\bar{\beta}}}\bigg(\frac{\mu_n}{Q^2x}\bigg), \hspace{1cm} x>0,
		\end{equation}
		where $Y_{\bs \alpha,\bs\beta}(x)$ and $X_{\bs\alpha,\bs{\bar{\beta}}}(x)$ are as defined as in \eqref{def:Y_alpha,beta} and \eqref{def:X_alpha,beta} respectively and $P_1(x)$ is a residual function given by 
		\begin{equation}
			P_1(x)= \frac{1}{2 \pi i}\int_{C} \G(s) F(s) x^{-s}ds.\label{def:res fn P_1(x)}	
		\end{equation} 
	\end{theorem}
	
	In the following theorem we show that the functional equation \eqref{fn eqn} can be obtained from this auxiliary modular relation \eqref{aux mod relation}. 
	\begin{theorem}\label{thm:modreln2 implies FE}
		The modular relation \eqref{aux mod relation} implies the functional equation \eqref{fn eqn}.
	\end{theorem} 
	
	The following Corollary follows from Theorems \ref{thm:FE implies mod rel1}, \ref{thm:mod rel1 implies FE}, \ref{thm:FE implies Rieszsum}, \ref{thm:riesz sum implies modreln 2}, and \ref{thm:modreln2 implies FE}. 
	\begin{corollary} \label{cor:equiv of identities}
		The functional equation \eqref{fn eqn}, modular relations \eqref{mod relation1} and \eqref{aux mod relation}, and the Riesz sum identity \eqref{riesz sum} are equivalent.
	\end{corollary}

We now discuss in more detail  recent literature related to these results.  In \cite{bochner}, Bochner defined Dirichlet series of Hecke-type as follows. 
Let $\lambda_n$ and $\mu_n$ be two strictly increasing sequences of real numbers tending to $\infty$ and  let ${a(n)}, {b(n)}$ be two sequences of complex numbers not identically zero. Consider the functions $\phi(s)$ and $\psi(s)$ representable as Dirichlet series
\begin{equation}
	\phi(s)=\sum_{n=1}^{\infty} \frac{a(n)}{\lambda_n^{s}} \hspace{4mm}\text{and}\hspace{4mm} \psi(s)=\sum_{n=1}^{\infty} \frac{b(n)}{\mu_n^{s}},
	\label{def:dirser}
\end{equation}
with finite abscissa of absolute convergence $\sigma_a$ and $\sigma_b$ respectively. 

These Dirichlet series satisfy the functional equation 
\begin{align}
	(2\pi)^{-s}	\Gamma(s)\phi(s)=(2\pi)^{s-\delta}\Gamma(\delta-s)\psi(\delta-s),
	\label{fe:bochner}
\end{align}
 for some $\delta>0$. That is, there exists a domain $D$ which is the exterior of a bounded closed set $S$, such that  in $D$ there exists a holomorphic function $\chi(s)$ satisfying 
\begin{align}
	\lim_{|t|\to\infty}\chi(\sigma+it)=0
\end{align}
uniformly in every bounded strip $-\infty<\sigma_1\leq\sigma\leq\sigma_2<\infty$, with  $\chi(s)=(2\pi)^{-s}	\Gamma(s)\phi(s)$ for all $\sigma>\sigma_a$, and $\chi(s)=(2\pi)^{s-\delta}\Gamma(\delta-s)\psi(\delta-s)$ for all $\sigma< {\delta- \sigma_b}$.
 Bochner proved that the functional equation \eqref{fe:bochner} is equivalent to the modular relation 
\begin{align}
	\sum_{n=1}^{\infty}a(n)e^{-\lambda_nx}=P(x)+\left(\frac{2\pi}{x}\right)^{\delta}	\sum_{n=1}^{\infty}b(n)e^{-4\pi^2\mu_n/x},\label{modrelsing}
\end{align}
where 
\begin{align}
	P(x)=\frac{1}{2\pi i}\int_{C}\chi(s)\left(\frac{2\pi}{x}\right)^s\,ds,
\end{align}
and $C$ denotes a curve (or curves)  in $D$ containing the set $S$. Chandrasekharan and Narasimhan \cite{ChandNarHeck} considered the same functional equation with arbitrary real $\delta$ and a more relaxed growth condition  on $\chi(s)$. In particular, they considered 
\begin{align}
	e^{-\epsilon|t|}\chi(\sigma+it)=O(1)
\end{align}
with $0<\epsilon<\pi/2$ as $|t|\to\infty$,	uniformly in every bounded strip $-\infty<\sigma_1\leq\sigma\leq\sigma_2<\infty$. They  proved the equivalence of the functional equation \eqref{fe:bochner} and the modular relation \eqref{modrelsing}. In the same paper, they also considered the identity 
\begin{align}
	S_{\rho}(x):=	\frac{1}{\Gamma(\rho +1)}\sideset{}{'}\sum_{\lambda_n\leq x}a(n)(x-\lambda_n)^{\rho}=\left(\frac{1}{2\pi}\right)^{\rho}\sum_{n=1}^{\infty}\left(\frac{x}{\mu_n}\right)^{\frac{1}{2}(\delta+\rho)}b(n)J_{\delta+\rho}(4\pi\sqrt{\mu_nx})+Q_{\rho}(x),
	\label{riiden}
\end{align}
where $J_{\d+\r}$ denotes the Bessel function of the first kind of order $\d+\r$, 
 $Q_{\rho}$ is a residual function defined by 
\begin{align}\label{def:Q rho}
	Q_{\rho}(x)=\frac{1}{2\pi i}\int_{C}\frac{\chi(s)(2\pi)^sx^{s+\rho}}{\Gamma(\rho+1+s)}\,ds,
\end{align}
and $C$ denotes a curve (or curves) in $D$ containing the set $S$. They proved that the  identity  \eqref{riiden} for $x>0$ and $\rho\geq 2\sigma_b-\delta-\frac{1}{2}$   is equivalent to the functional equation \eqref{fe:bochner}. In \cite{ChandrasekharanI}, the arithmetic identity \eqref{riiden} served as the  main tool for Chandrasekharan and Narasimhan to obtain the omega-bound 
\begin{align}
	\Re[S_{\rho}(x)-Q_{\rho}(x)]=\Omega_{\pm}\left(x^{\frac{\delta}{2}+\frac{\rho}{2}-\frac{1}{4}}\right),
\end{align}
if  $\Re b(n)\neq 0$  for at least one value of $n$ and 
\begin{align}
	\Im[S_{\rho}(x)-Q_{\rho}(x)]=\Omega_{\pm}\left(x^{\frac{\delta}{2}+\frac{\rho}{2}-\frac{1}{4}}\right),
\end{align}
if $\Im b(n) \neq 0$ for at least one value of $n$.

Bochner (p. 341, \cite{bochner}) made the following observation about  the modular relation \eqref{modrelsing}. If the function $\chi(s)$ has simple poles at $s=0$ and $s=\delta$ then the residual function $P(x)$ can be written in the form $-a(0)+b(0)x^{-\delta}$ for some $a(0), b(0) \in \mathbb{C}$. In this case the modular relation \eqref{modrelsing} can be re-written as 
\begin{align} 	\sum_{n=0}^{\infty}a(n)e^{-\lambda'_nx}={x}^{-\delta}	\left( \sum_{n=0}^{\infty}b(n)e^{-\mu'_n/x} \right),
\end{align} 
where  $\lambda'_n={2\pi} {\lambda_n}$ and $\mu'_n={2\pi}{\mu_n}$ and $\lambda'_0 =0 =\mu_0'$. 
 This modular relation implies the summation formula 
\begin{align}
	\sum_{n=0}^{\infty}a(n)f(\lambda'_nx)={x}^{-\delta}	\sum_{n=0}^{\infty}b(n)g(\mu'_n/x),\label{modrelsym}
\end{align}
for some pair of functions $f,g$ which are connected by the Hankel transformations 
\begin{align}
	g(t)=t^{-\frac{1}{2}(\delta-1)}\int_0^{\infty}J_{\delta-1}(2\sqrt{tx})x^{\frac{1}{2}(\delta-1)}f(x)\,dx, 
	\\	f(x)=x^{-\frac{1}{2}(\delta-1)}\int_0^{\infty}J_{\delta-1}(2\sqrt{tx})t^{\frac{1}{2}(\delta-1)}f(t)\,dt.
\end{align}
If one chooses $x=1$ and the support of $f$ is compact,  then \eqref{modrelsym} yields  analogues of the Vorono\"{\i} summation formula \eqref{eq:vsr}. Moreover, if $\rho\geq 0$ (for instance if   $2\sigma_b\geq \delta+ \frac{1}{2}$) in the identity \eqref{riiden},  then analogues of the Vorono\"{\i} summation formula can be obtained for $\sum a(n)f(n)$ by an  argument similar  to that discussed for  \eqref{partsum}.

In \cite{BerndtI}, Berndt considered the  Dirichlet series \eqref{def:dirser} and  allowed for  higher powers of the Gamma function in the functional equation. More precisely,  he defined the holomorphic function $\chi(s)$ by $\chi(s)=	\Gamma^m(s)\phi(s)$ for all $\sigma>\sigma_a$ and $\chi(s)=\Gamma^m(\delta-s)\psi(\delta-s)$ for all $\sigma<\delta-\sigma_b$. Let $p$ be the smallest positive integer so that  $\gamma= \sigma_b+p-\frac{1}{4m}>\max(0,\sigma_a,\sigma_b)$. Berndt showed that if 
\begin{align}
	\chi(s)=O\left(e^{\exp\left(\frac{ \theta \pi |s|}{2\gamma-\d}\right)}\right),
\end{align}
for some $\theta<1$, uniformly in the half-strip $\delta-\gamma<\sigma<\gamma$, $|t|\geq \eta$ contained in $D$, then we have an equivalence between  the functional equation $\chi(s)=\chi(\d-s)$ and the identities 
\begin{align}
	\sum_{n=1}^{\infty}a(n)E_m(\lambda_nx)=P(x)+x^{-\delta}	\sum_{n=1}^{\infty}b(n)E_m(\mu_n/x), 
	\quad (x>0) \label{modrelberndt}
\end{align}
and 
\begin{align}
	\frac{1}{\Gamma(\rho +1)}\sideset{}{'}\sum_{\lambda_n\leq x}a(n)(x-\lambda_n)^{\rho}=2^{\rho(1-m)}\sum_{n=1}^{\infty}\left(\frac{x}{\mu_n}\right)^{\frac{1}{2}(\delta+\rho)}b(n)K_{\delta+\rho}(2^m\sqrt{\mu_nx};  \d-1; m)+Q_{\rho}(x),
\end{align}
for $\delta>-1/2$ and $\rho>2m\sigma_b-m\delta-1/2$.  Here 
\begin{align}
	E_m(x)=\frac{1}{2\pi i}\int_{(c)}\Gamma^m(s)x^{-s}\,ds, \quad c>0, x>0,
\end{align}
\begin{align}
	K_{\nu}(x;\mu;m)=\frac{1}{2\pi i } \int_{(c)}\frac{2^{2ms-(m-1)\mu-\nu-m+1}\Gamma(s)x^{\nu-2s}}{\Gamma^{m-1}(\mu+1-s)\Gamma(\nu+1-s)}\,ds,\quad  0<c< \frac{1}{2}\min (\mu,\nu)+\frac{3}{4},
\end{align}
\begin{align}
	P(x)=\frac{1}{2\pi i}\int_{C}\chi(s)x^{-s}\,ds 
\text{ and  }
	Q_{\rho}(x)=\frac{1}{2\pi i}\int_{C}\frac{\Gamma(s)\phi(s)x^{s+\rho}}{\Gamma(\rho+1+s)}\,ds,
\end{align}
where $C$ denotes a curve (or curves) in $D$ containing the set $S$. 

In a series of papers, Berndt \cite{BerndtGen}, \cite{BerndtIII}, \cite{BerndtIV}, \cite{BerndtV}, \cite{BerndtVII}  presented a detailed study of Riesz-type summation formulae for various classes of Dirichlet series. In \cite{BerndtKimZaharescuIV2012}, Berndt, Kim and Zaharescu studied various Riesz-type sums and then in \cite{KimSun}, Kim  gave Riesz-type identities for weighted sums of divisor functions. In recent work by Berndt, Dixit, Gupta and Zaharescu \cite{BDGZI}, \cite{BDGZII}, they obtained new identities of the type \eqref{riiden} involving the modified Bessel function $K_{\nu}(z)$ of order $\nu$.

In \cite[Theorem 6.1]{BDRZ}, the first author along with Berndt, Dixit and Zaharescu proved a Vorono\"{\i} summation formula analogous to \eqref{eq:vsd} for the divisor function $\sigma_s(n)$. Subsequently,  Dixit, Maji and Vatwani \cite[Theorem 2.2]{DixitMajiVatwani} obtained a  version of the Vorono\"{\i} summation formula for a more general divisor function $\sigma_s^{(k)}(n)$, which gives \cite[Theorem 6.1]{BDRZ} as a corollary but for a larger domain. Again in \cite{BanerjeeMaji2023}, Banerjee and Maji obtained a summation formula for $\sigma_s^{(k)}(n)$ involving the modified Bessel function of the second kind.

In the literature, there are many studies relevant to obtaining the modular-type relation \eqref{modrelsing} from the functional equation in various settings. For instance, Arai, Chakraborty, Kanemitsu \cite{AraiChakrabortyKanemitsu} and Chakraborty, Kanemitsu, Maji \cite{ChakrabortyKanemitsuMaji} obtained modular-type relations associated to Dedekind zeta functions and the Rankin-Selberg $L$-function respectively. A breakthrough result of Tsukada \cite{tsukada} gives a more generalized modular-type identity for a class of Dirichlet series. This general result allows one to recover many existing modular-type identities as special cases, for instance \eqref{modrelsing} and \eqref{modrelberndt}. Further in \cite{kanemitsu}, \cite{KanemitsuTanigawaTsukada}, Kanemitsu, Tanigawa and Tsukada studied these kinds of generalized modular relations along with many number theoretic applications. In \cite{KanemitsuTanigawaTsukada}, they established the equivalence between the functional equation and modular relation in a very general setting. For a more detailed description of these generalized modular relations, we refer the reader to the book \cite{Kanemitsu2015book} by Kanemitsu and Tsukada.


	\section{Prerequisites}
We define a  residual function as done by Bochner in  \cite{bochner}.
	\begin{definition}  \cite[Definition 2]{bochner} \label{def:res fn}
		A function $P(x)$ is said to be a residual function if
		\begin{enumerate}[(i)]
			\item $P(x)$ is defined and differentiable in $(0, \infty)$. Moreover, $P(x)=O(x^{-c})$ as $x \rightarrow 0$ and $P(x)=O(x^{c})$ as $x \rightarrow \infty$, for some constant $c>0$, so that the functions 
			\begin{equation}
				I_1(s)= \int_{0}^{1} P(x)x^{s-1}dx, \hspace{2cm} 	I_2(s)=- \int_{1}^{\infty} P(x)x^{s-1}dx,
			\end{equation}
		can be introduced in some right half-plane and  left half-plane respectively.
		\item $I_1(s)$  and $I_2(s)$ can be continued into each other in a domain $D$ as introduced in Definition \ref{def:fn eqn of F}.
		\item $\lim_{|t|\rightarrow \infty}I(\sigma+it)=0$, uniformly in $-\infty<\sigma_1 \leq\sigma\leq\sigma_2< +\infty$, where $I$ denotes the function obtained by analytic continuation in (ii).  	
	\end{enumerate}
	\end{definition}

The following example of a residual function was given by Bochner \cite{bochner}.  
	\begin{lemma}\cite[Lemma 1] {bochner} \label{lem:residuefn}
	Suppose $\chi(s)$ is analytic in a domain $D$ in $\mathbb C$. Then the integral
	\begin{equation}
		P(x)= \frac{1}{2 \pi i}\int_{C} \chi(s)x^{-s}ds,
	\end{equation}
over a bounded curve or curves $C$ in $D$, with $x^{-s}=\exp(-s\log x)$,
	  is a residual function.	
	\end{lemma}

We will need the following information about the integral $Q_\rho(x)$. 
\begin{lemma}\label{lem:Q_rho is C infinite}
	The residual function $Q_\r(x)$, defined in \eqref{resfnQ_rho} is a $C^{\infty}$ function.
\end{lemma}
\begin{proof}
	As $C$ is a bounded curve or  a  finite union of bounded curves, it contains a finite number of poles  of the integrand. By the  Cauchy residue theorem, $Q_\r(x)$ is the sum of  residues at these singularities, and is thus a function of the form $\sum c_\alpha x^{\alpha}$ . Since this a finite sum, $Q_\r(x)$ is a  $C^{\infty}$  function, as needed. 
\end{proof}

The following lemma is a version of the  Phragm\'en-Lindel\"of principle due to J. E. Littlewood \cite{Littlewoodtheoryoffn}. We state the version given by Chandrasekharan and Narasimhan \cite{ChandNarHeck}.
	\begin{lemma} \cite[Lemma 2]{ChandNarHeck} \label{lem:phragmen-lind}
Suppose that  $f(s)$,  $s=\sigma+it$,  is regular in a half-strip $S$ defined by $a<\s<b$, $|t|>c$ and continuous on the boundary. Moreover if we have  
\begin{equation}
	f(s)=O\left(\exp\left[\exp\left(\frac{\theta \pi|s|}{b-a}\right)\right]\right),
\end{equation}
uniformly in $S$ for $\theta <1$, and $f=o(1)$ on $x=a$, and on $x=b$, then $f=o(1)$ uniformly in $S$. 
	\end{lemma}

The Gamma function is the Mellin transform of $e^{-x}$, defined for any $s \in \mathbb{C} $ with  $\Re(s) > 0$ by the real integral 
	$$ \Gamma (s) = \int_{0}^{\infty} x^{s-1} e^{-x} dx. $$
 Similarly, for any $\alpha, \beta \in \mathbb C$ with $\Re(\a s+\b)>0$,  we can express $\G(\a s+\b)$ as  
	\begin{equation} \label{def:gamma fn}	
		\G(\a s+\b)=\int_{0}^{\infty} f_{\a, \b}(x) x^{s-1} dx, 
	\end{equation}
	where 
	\begin{align}
		\label{def:f alpha beta}
	f_{\a, \b}(x)= \frac{\exp(-x^{1/\a})}{\a x^{\b/\a}}.
	\end{align}

We can extend this to a product of Gamma functions, for instance, as in Berndt \cite[Lemma 5]{BerndtI}. 

\begin{lemma} \label{lem:multiple mellin int}
	Define for $x, a>0$, 
	$$
	Y_{\bs \alpha,\bs\beta}(x):=\frac{1}{2 \pi i}\int_{a-i\infty}^{a+i\infty} \G(s)\prod_{i=1}^{r} \Gamma(\alpha_i s+ \beta_i) x^{-s}ds,
	$$ where $\bs \a =(\a_1, \dots, \a_r),\, \bs \b=(\b_1, \dots, \b_r)$.
	Then we have, 
	\begin{equation}
		Y_{\bs \alpha,\bs\beta}(x)=\int_{0}^{\infty}f_{\a_r, \b_r}(u_r) \frac{du_r}{u_r} \cdots \int_{0}^{\infty}f_{\a_1, \b_1}(u_1) \exp\left(\frac{-x}{u_1\cdots u_r}\right) \frac{du_1}{u_1},
	\end{equation}
where $f_{\a,\b}(x)$ is as defined in \eqref{def:f alpha beta}. 
\end{lemma}
\begin{proof}
	We recall the following  result on multiple Mellin integrals \cite[p.53]{TitchmarshFourier}.
If $F(s)$, $F_1(s)$, \dots, and $F_n(s)$ are Mellin transforms of $f(x)$, $f_1(x)$, \dots, and $f_n(x)$ respectively then 
	\begin{equation}
		\frac{1}{2 \pi i}\int_{a-i\infty}^{a+i\infty}F(s)F_1(s)\cdots F_n(s)   x^{-s}ds=\int_{0}^{\infty}f_n(u_n) \frac{du_n}{u_n} \cdots \int_{0}^{\infty}f_1(u_1) f\left(\frac{x}{u_1\cdots u_n}\right) \frac{du_1}{u_1}.
	\end{equation}
	Combining this with \eqref{def:gamma fn} completes the proof. 
\end{proof}


	We recall an integral transformation formula given by Chandrasekharan and Narasimhan on p. 35 of \cite{ChandNarApprox}.
	\begin{lemma}\label{lem:inv mel of Gammacos}
	 For any $c<-\frac{1}{2}, c\notin \mathbb{Z}$ and $\a \in \mathbb{R}$, we have
		\begin{equation}
			\frac{1}{2\pi i}\int_{c-i\infty}^{c+i\infty}\G(s)\cos\left(\frac{\pi}{2}s+\a \right)x^{-s}ds=\cos(x+\a)-\sum_{0\leq n< |c|} (-1)^n \frac{x^n}{n!}\cos\left(\a-\frac{\pi}{2}n \right).
		\end{equation}
	\end{lemma}
We will record here a version of Stirling's asymptotic formula for $\log{\G(s)}$ (refer to the equation following (8) in \cite{ChandNarApprox}). 
\begin{lemma}\label{lem:gamma asymp}
	For any $c\in \mathbb{C}, |\arg s|<\pi$, we have as $|s|\to \infty$,
	\begin{equation}\label{eq:stirling for log gamma}
		\log \G(s+c)= 
		\left(s+c- \frac{1}{2}\right)\log s-s+ \frac{\log{2\pi}}{2}+ \sum_{n=1}^{m}C_n s^{-n}+O\left(\frac{1}{|s|^{m+1}}\right),
	\end{equation} where $m$ is any positive integer and $C_n$ are some constants depending on $c$.
\end{lemma}

Note that for $s=\sigma+it$, if  $\sigma$ is fixed and 	 $|t|\rightarrow \infty$, this yields 
\begin{equation} \label{eq:stirling}
	|\Gamma(s)|\sim e^{- \frac{\pi}{2}|t|}|t|^{\sigma-\frac{1}{2}}\sqrt{2\pi}. 
\end{equation}

The following lemma is a form of Perron's formula, given by Chandrasekharan and Minakshisundaram \cite{Typicalmeans}.
\begin{lemma} \cite[Lemma 3.65]{Typicalmeans} \label{lem:perronsformula}
	If $f(s)=\sum_{n=1}^{\infty}a_n \lambda_n^{-s}$, with $\sum_{n=1}^{\infty}|a_n| \lambda_n^{-\alpha}< \infty$, then for $\r \geq 0$, $\sigma>0$, and $\sigma\geq \alpha$, we have
	$$\frac{1}{\Gamma(\rho+1)}\sideset{}{'}\sum_{\lambda_n \leqslant x}{} a_{n}(x-\lambda_n)^\rho=\frac{1}{2\pi i} \int_{\sigma-i \infty}^{\sigma+i \infty} \frac{f(s)\Gamma(s) x^{s+\rho}}{\Gamma(s+\rho+1)}ds,$$
	the dash indicating that the last term of the sum has to be multiplied by $\frac{1}{2}$ if $\r=0$ and $x=\lambda_n$. 
\end{lemma}

In the following lemma, we provide a growth estimate for $Z_{\bs \alpha,\bs\beta}(x)$, defined in \eqref{def:Z_alpha,beta}. 	
	
\begin{lemma}\label{lem:exp order of Z_aplha,beta}
	For $x>0$, we have 
	\begin{equation}
		Z_{\bs \alpha,\bs\beta}(x)\ll \exp{\left(-cx^{\frac{1}{d_F^{'}}}\right)},
	\end{equation}
as $x\rightarrow \infty $, where $c>0$ is a constant depending on $\bs \a$ and $\bs \b$, and  $d_F^{'}=\sum_{i=1}^{r}\a_i$.
\end{lemma}	
\begin{proof}
	For some sufficiently large $T>0$, let us write 
	\begin{align}
		\frac{1}{2 \pi i}\int_{a-i\infty}^{a+i\infty} \prod_{i=1}^{r} \Gamma(\alpha_i s+ \beta_i) x^{-s}ds&=\frac{1}{2 \pi i}\left[\int_{a-i\infty}^{a-iT}+\int_{a-iT}^{a+iT}+\int_{a+iT}^{a+i\infty}\right]\prod_{i=1}^{r} \Gamma(\alpha_i s+ \beta_i) x^{-s}ds  \label{integralof gamma on line} \\
		&=\frac{1}{2 \pi i}\left[I_1+I_2+I_3\right], \qquad (\text{say}).
	\end{align}
Now take $N>a$ and consider a rectangle $R$ with vertices $a\pm iT$ and $N\pm iT$. We shall choose $N$ sufficiently large later, depending upon $x$. For $\Re(s) \geq a$, as $\Re(\alpha_i s+ \beta_i)>0$ for each $i$, the Gamma factors $\Gamma(\alpha_i s+ \beta_i)$ have no singularities inside the rectangle $R$,  so the integrand is analytic inside $R$. By Cauchy's residue theorem, we have
\begin{align}
	I_2=\int_{a-iT}^{a+iT}\prod_{i=1}^{r} \Gamma(\alpha_i s+ \beta_i) x^{-s}ds&= \left[\int_{a-iT}^{N-iT}+\int_{N-iT}^{N+iT}+\int_{N+iT}^{a+iT}\right]\prod_{i=1}^{r} \Gamma(\alpha_i s+ \beta_i) x^{-s}ds
	 \label{integral on contour1} \\
	 &=I_{21}+I_{22}+I_{23}, \qquad (\text{say}).
\end{align}
We will first estimate $I_{22}$ using estimates for the growth of the Gamma function.
 Writing $s=\s+it$, from Lemma \ref{lem:gamma asymp}, we have
\begin{equation}\label{gen asymp for gamma}
	|\G(\s+it)|\ll e^{-\s}\left|s^{s-\frac{1}{2}}\right|=e^{-\s} |s|^{\s-\frac{1}{2}}e^{-t\arg s},
\end{equation} since 
\begin{equation}
\left|s^{s-\frac{1}{2}}\right|= e^{\Re\left(\left(s-\frac{1}{2}\right)\log s\right)}=e^{\left(\s-\frac{1}{2}\right)\log |s|-t \arg s},
\end{equation} where the principal branch of the logarithm in $\mathbb{C}\setminus \left(\left. \infty,0\right. \right]$ is chosen.
Now observe that $-t\arg s \leq 0$ when  $-\pi < \arg s < \pi$, so \eqref{gen asymp for gamma} yields the bound
\begin{equation}\label{modt<sigma}
	|\G(\s+it)|\ll e^{-\s} |s|^{\s-\frac{1}{2}}.
\end{equation}
 Whenever $s$ satisfies $|t|>\s>0$, we have $|\arg s|>\frac{\pi}{4}$ and in this case, \eqref{gen asymp for gamma} gives the better bound 
\begin{equation} \label{modt>sigma}
		|\G(\s+it)|\ll e^{-\s} |s|^{\s-\frac{1}{2}}e^{-|t|\frac{\pi}{4}}.
\end{equation}

Taking  $\b_i=b_i+ic_i$, where $b_i,c_i \in \mathbb{R}$, we can use the bound \eqref{modt>sigma} for $\G(\a_is+\b_i)$ whenever 
$|\Im(\a_is+\b_i)|>\Re(\a_is+\b_i)$, that is when  $t>\s+\frac{b_i-c_i}{\a_i}$ or $t<-\s-\frac{b_i+c_i}{\a_i}$.
Now, let us consider $T_1=\max_{i=1}^{r}{\left[N+\frac{b_i-c_i}{\a_i}\right]},\, T_2=\min_{i=1}^{r}{\left[-N-\frac{b_i+c_i}{\a_i}\right]} $.
For $N$ sufficiently large, we have $T_1>0$ and $T_2<0$. We write
\begin{equation}
	\label{eq:3integrals}
I_{22}=\int_{N-iT}^{N+iT}\prod_{i=1}^{r} \Gamma(\alpha_i s+ \beta_i) x^{-s}ds=\left[\int_{-T}^{T_2}+\int_{T_2}^{T_1}+\int_{T_1}^{T}\right]\prod_{i=1}^{r} \Gamma(\alpha_i(N+it)+ \beta_i) x^{-(N+it)}idt.
\end{equation}
We will use \eqref{modt>sigma} to estimate the first and third integral on the right-hand side above. We get
\begin{multline}
\int_{T_1}^{T}\prod_{i=1}^{r}\left|\Gamma(\alpha_i(N+it)+ \beta_i)\right| x^{-N}dt \label{step1from T_1 to T}\\
 \ll x^{-N}e^{-\sum_{i=1}^{r}(\a_i N+b_i)}\int_{T_1}^{T} e^{-\frac{\pi}{4}\sum_{i=1}^{r}|\a_i t+c_i|}\prod_{i=1}^{r}\left|(\a_i N+b_i)+i(\a_i t+c_i)\right|^{\a_i N+b_i-\frac{1}{2}} dt.
\end{multline}
Letting $\g=\max_{i=1}^{r}{\left\{|\a_i|,|b_i|,|c_i|\right\}}$, we have $\left|\a_i N+b_i+i(\a_i t+c_i)\right| \leq \g(N+t+2)$, keeping in mind that $t>0$. Using the notation $d_F^{'}=\sum_{i=1}^{r}\a_i$, \eqref{step1from T_1 to T} is
\begin{equation}
	\ll_{b_i,c_i,r} x^{-N} e^{-Nd_F^{'}} \g^{Nd_F^{'}} \int_{T_1}^{T} e^{-\frac{\pi}{4}td_F^{'}} (N+t+2)^{Nd_F^{'}+\sum_{i=1}^{r}b_i-\frac{r}{2}}dt.
\end{equation}
We now make the change of variable $u=N+t+2$ to see that the above expression is 
\begin{align}
	&\ll_{b_i,c_i,r,d_F^{'}} x^{-N} e^{-Nd_F^{'}} \g^{Nd_F^{'}} \int_{N+T_1+2}^{N+T+2} e^{-\frac{\pi}{4}ud_F^{'}}e^{\frac{\pi}{4}Nd_F^{'}} u^{\left(Nd_F^{'}+\sum_{i=1}^{r}b_i-\frac{r}{2}\right)} du \\
	&\ll_{\a_i,b_i,c_i,r} x^{-N} e^{-Nd_F^{'}\left(1-\frac{\pi}{4}\right)} \g^{Nd_F^{'}} \int_{N+T_1+2}^{N+T+2} e^{-\frac{\pi}{4}ud_F^{'}} u^{\left(Nd_F^{'}+\sum_{i=1}^{r}b_i-\frac{r}{2}\right)} du. \label{step2from T_1 to T} 
	\end{align}
Putting $v=\frac{\pi}{4}ud_F^{'}$, the integrand above is $$e^{-v}\left(\frac{v}{\frac{\pi}{4}d_F^{'}}\right)^{Nd_F^{'}+\sum_{i=1}^{r}b_i-\frac{r}{2}} \ll e^{-v} v^{Nd_F^{'}+\sum_{i=1}^{r}b_i-\frac{r}{2}}, $$
so that \eqref{step2from T_1 to T} is
\begin{align}
	&\ll_{\a_i,b_i,c_i,r} x^{-N} e^{-Nd_F^{'}\left(1-\frac{\pi}{4}\right)} \g^{Nd_F^{'}} \Gamma\left(Nd_F^{'}+\sum_{i=1}^{r}b_i-\frac{r}{2}+1\right)\\
	&\ll_{\a_i,b_i,c_i,r} x^{-N} e^{-Nd_F^{'}\left(1-\frac{\pi}{4}\right)} \g^{Nd_F^{'}} \left(Nd_F^{'}+\sum_{i=1}^{r}b_i-\frac{r}{2}+1\right)^{Nd_F^{'}+\sum_{i=1}^{r}b_i-\frac{r}{2}+\frac{1}{2}}e^{-Nd_F^{'}},
\end{align}
using the estimate \eqref{modt<sigma}. For $N$ sufficiently large, we can write $Nd_F^{'}+\sum_{i=1}^{r}b_i-\frac{r}{2}+1\leq CNd_F^{'}$, for some $C>0$. We also have $\sum_{i=1}^{r}b_i-\frac{r}{2}+\frac{1}{2}\leq C_1$ for some constant $C_1$ depending on $\bs \a, \bs \b$. Thus, the above expression is 
\begin{align}
	&\ll_{\bs \a,\bs \b} x^{-N} e^{-Nd_F^{'}\left(2-\frac{\pi}{4}\right)} \g^{Nd_F^{'}}\left(CNd_F^{'}\right)^{Nd_F^{'}} \left(CNd_F^{'}\right)^{C_1}\\
	&\ll_{\bs \a,\bs \b} x^{-N} e^{-Nd_F^{'}\left(2-\frac{\pi}{4}\right)} (C\g)^{Nd_F^{'}}\left(Nd_F^{'}\right)^{Nd_F^{'}+C_1} \label{step3from T_1 to T}\\
	&\ll \exp \left(-Nd_F^{'}\left[\frac{\log x}{d_F^{'}}+2-\frac{\pi}{4}-\log C\g -\left(1+\frac{C_1}{Nd_F^{'}}\right)\log Nd_F^{'}+O_{\bs \a, \bs \b}\left(\frac{1}{Nd_F^{'}}\right)\right]\right). \label{finalstepfrom T_1 to T}
\end{align}
It can be checked that in order to ensure that the expression in square brackets is positive as $x\rightarrow \infty$, it is enough to choose 
\begin{equation}
	N=\frac{A}{ d_F^{'}}x^{\frac{1}{d_F^{'}}}, \label{eq:choice of N}
\end{equation} for some $A>0$ satisfying $\log A<2-\frac{\pi}{4}-\log C\g$.
This means that \eqref{finalstepfrom T_1 to T} is
\begin{equation}
	\ll \exp(-C_2x^{\frac{1}{d_F^{'}}}), \label{eq:estimationT_1 to T}
\end{equation}
for some positive constant $C_2$, depending upon $\bs \a, \bs \b$. In a similar way, we also get
\begin{equation}
	\int_{-T}^{T_2}\prod_{i=1}^{r} \Gamma(\alpha_i(N+it)+ \beta_i) x^{-(N+it)}idt \ll \exp(-C_3x^{\frac{1}{d_F^{'}}}), \label{eq:estimation-T to T_2}
\end{equation}
for some positive  constant $C_3 = C_3(\bs \a, \bs \b)$.

We now turn to the middle integral in \eqref{eq:3integrals}. Using the bound \eqref{modt<sigma}, we get
\begin{multline}
	\int_{T_2}^{T_1}\prod_{i=1}^{r}\left|\Gamma(\alpha_i(N+it)+ \beta_i)\right| x^{-N}dt\\
	\ll x^{-N}e^{-\sum_{i=1}^{r}(\a_i N+b_i)}\int_{T_2}^{T_1} \prod_{i=1}^{r}\left|(\a_i N+b_i)+i(\a_i t+c_i)\right|^{\a_i N+b_i-\frac{1}{2}} dt. \label{step1from T_2 to T_1}
\end{multline}
Taking $\g$ as defined previously, we have $\left|\a_i N+b_i+i(\a_i t+c_i)\right| \leq \g(N+|t|+2)$. With the notation $d_F^{'}$ as before, \eqref{step1from T_2 to T_1} is
\begin{equation}
	\ll_{b_i,r} x^{-N} e^{-Nd_F^{'}} \g^{Nd_F^{'}} \int_{T_2}^{T_1}(N+|t|+2)^{Nd_F^{'}+\sum_{i=1}^{r}b_i-\frac{r}{2}} dt.
\end{equation}
As  $T_1,T_2$ are both $\ll N$, we have $|t| \ll N$ in the domain of the above integral. Hence, for some $C>0$, the above expression is
\begin{align}
	&\ll_{b_i,r} x^{-N} e^{-Nd_F^{'}} \g^{Nd_F^{'}} (CN)^{Nd_F^{'}+\sum_{i=1}^{r}b_i-\frac{r}{2}}\int_{T_2}^{T_1} dt\\
	&\ll_{b_i,r} x^{-N} e^{-Nd_F^{'}} (C\g)^{Nd_F^{'}}  N^{Nd_F^{'}+\sum_{i=1}^{r}b_i-\frac{r}{2}+1}.
\end{align}
Similar to as done for \eqref{step3from T_1 to T}, the choice $N=Bx^{\frac{1}{d_F^{'}}}$, for some $B>0$ satisfying $\log B <1-\log C\g$, yields 
\begin{equation}
	\int_{T_2}^{T_1}\prod_{i=1}^{r}\Gamma(\alpha_i(N+it)+ \beta_i) x^{-N}dt  \ll \exp(-C_4x^{\frac{1}{d_F^{'}}}), \label{eq:estimationT_2 to T_1}
\end{equation} for some positive constant $C_4= C_4(\bs \a, \bs \b)$. Finally, we choose $N=Dx^{\frac{1}{d_F^{'}}}$, where $D= \min \{A, B\}$ and combine \eqref{eq:3integrals}, \eqref{eq:estimationT_1 to T}, \eqref{eq:estimation-T to T_2} and \eqref{eq:estimationT_2 to T_1}, to get
\begin{equation}
	I_{22}\ll \exp(-Cx^{\frac{1}{d_F^{'}}}),
\end{equation}
for some $C=C(\bs \a, \bs \b)>0$. Turning to $I_{23}$, we have
\begin{align}
	I_{23}=\int_{N+ iT}^{a + iT}\prod_{i=1}^{r} \Gamma(\alpha_i s+ \beta_i)& x^{-s}ds 
	\ll \int_{a}^{N}\prod_{i=1}^{r}\left|\Gamma(\alpha_i(\s+iT)+ \beta_i)\right| x^{-\s}d\s\\
	&\ll e^{-\frac{\pi}{4}T}\!\int_{a}^{N}\!\!\! e^{-\sum_{i=1}^{r}(\a_i \s+b_i)}\!\prod_{i=1}^{r}\left|(\a_i \s+b_i)+i(\a_i T+c_i)\right|^{\a_i \s+b_i-\frac{1}{2}}x^{-\s} d\s,
\end{align}upon using the bound \eqref{modt>sigma}. Using the same notations $\g, d_F^{'}$ as before, the above is
\begin{equation}
	\ll_{b_i,r} e^{-\frac{\pi}{4}T} e^{-ad_F^{'}}\g^{Nd_F^{'}}(N+T+2)^{Nd_F^{'}+\sum_{i=1}^{r}b_i-\frac{r}{2}}x^{-a}N,\label{finalhorizontalibdd}
\end{equation} considering the upper bound of the integrand in the interval $a \leq \s \leq N$ for $x>1,T>0$ and sufficiently large $N$. If we take $T\gg N^2$, due to the dominance of the exponential term in  \eqref{finalhorizontalibdd}, the contribution from this horizontal integral is negligible. In a similar manner, we can also bound $I_{21}$.

We now consider the integrals $I_1$ and $I_3$. We write
\begin{align}
	I_3=\int_{a+iT}^{a+i\infty}\prod_{i=1}^{r} \Gamma(\alpha_i s+ \beta_i) x^{-s}ds &\ll 
	\int_{T}^{\infty}\prod_{i=1}^{r}\left|\Gamma(\alpha_i(a+it)+ \beta_i)\right| x^{-a}dt \\
	&\ll \int_{T}^{\infty}e^{-\frac{\pi}{2}\sum_{i=1}^{r}|\a_it+c_i|}\prod_{i=1}^{r} |\a_it+c_i|^{\a_ia+b_i-\frac{1}{2}}x^{-a}dt,
\end{align}
using the Stirling formula \eqref{eq:stirling}.
For sufficiently large $T$, the above integral becomes negligible due to exponential decay, as seen for \eqref{finalhorizontalibdd}. The same applies to $I_1$. Thus, we conclude that
\begin{equation}
	Z_{\bs \alpha,\bs \beta}(x)\ll \exp{\left(-cx^{\frac{1}{d_F^{'}}}\right)},
\end{equation} where $c$ is a positive constant that depends upon $\bs \a, \bs \b$.
\end{proof}

As a corollary to the above lemma, we have the following estimate for the integral $Y_{\bs \alpha,\bs\beta}(x)$, which was defined in \eqref{def:Y_alpha,beta}. 
\begin{corollary} \label{cor:exp order of Z_aplha,beta}
 We have as  $x\rightarrow \infty$,
 \begin{equation}
 Y_{\bs \alpha,\bs\beta}(x) \ll \exp{\left(-cx^{\frac{1}{1+d_F^{'}}}\right)},
 \end{equation}

 where $d_F^{'}=\sum_{i=1}^{r}\a_i$ and $c$ is some positive constant depending on $\bs \a$ and $\bs \b$.
\end{corollary}

In the following lemma, we prove an asymptotic formula for an integral involving multiple Gamma factors. Our result is different from those existing in the literature, such as Lemma 1 of Chandrasekharan and Narasimhan \cite{ChandNarApprox}, since the functional equation considered by us entails conjugates of $\b_i$ appearing in the Gamma factors in the numerator. Moreover, our result is not restricted to integer values of $\r$ and holds for any large real value of $\r$. 
\begin{lemma} \label{lem:asym for I_rho}
	Let  $\d \geq 0, d_F^{'}=\sum_{i=1}^{r}\a_i, a>0$ be sufficiently large  and $\r>(2a-\d)d_F^{'}-1$. Consider
	\begin{equation}
		I_\r(x):=\frac{1}{2 \pi i}\int_{a-i\infty}^{a+i\infty} \frac{\Gamma(\delta-s) \prod_{i=1}^{r}\Gamma(\alpha_is+\bar{\beta}_i)}{\Gamma(1+\delta-s+\rho)\prod_{i=1}^{r}\Gamma(\alpha_i(\delta-s)+\beta_i)}x^{\r+\d-s}ds.
	\end{equation}
	Then for any positive integer $m$, we have as $x\rightarrow \infty$,
	\begin{equation}
		I_\r(x)=\sum_{n=0}^{m} A_n x^{\frac{\omega'-n-2i\nu-\frac{1}{2}}{2d_F^{'}}}
		 \cos\left((h^{-1}x)^{\frac{1}{2d_F^{'}}}+\pi\left(d_F^{'}\gamma +\frac{n}{2}
		+ i\nu- \mu\right)\right) 
		+ O\bigg( x^{\frac{\omega'-m-\frac{3}{2}}{2d_F^{'}}}\bigg),
	\end{equation}
	where $A_n$ are some constants, and writing $\b_i=b_i+ic_i$, we have
	\begin{align} 
		&\nu=\sum_{i=1}^{r}c_i, \hspace{0.5mm} \omega'=d_F^{'}\d+(2d_F^{'}-1)\r,\hspace{0.5mm} h=\exp \left\{\sum_{i=1}^{r}2\a_i\log\a_i-2d_F^{'} \log2d_F^{'}\right\},\hspace{0.5mm} \mu=\frac{1}{2}+\sum_{i=1}^{r}\left(\beta_i-\frac{1}{2}\right),	\\
		&\g=-\left(\frac{\d}{2}+\frac{\r}{2d_F^{'}}+\frac{1}{4d_F^{'}}\right),\hspace{2mm}
		k=\sqrt{2\pi}(2d_F^{'})^{-2d_F^{'}\g+2i\nu+\frac{1}{2}} \exp{\left(-\sum_{i=1}^{r}(\a_i\d+2ic_i)\log \a_i\right)}. \label{notations in lem for I_rho}\\ 
	\end{align}
\end{lemma}
\begin{proof}
		Let us define 
	\begin{equation}\label{def of G_rho}
		G_\r(s):=\frac{\Gamma(\delta-s) \prod_{j=1}^{r}\Gamma(\alpha_js+\bar{\beta}_j)}{\Gamma(1+\delta-s+\rho)\prod_{j=1}^{r}\Gamma(\alpha_j(\delta-s)+\beta_j)}.
	\end{equation}
Taking logarithm on both the sides above, we write
\begin{equation}\label{eq:log of G_rho}
	\log{G_\r(s)}=\log{\Gamma(\delta-s)}+\sum_{j=1}^{r}\log{\Gamma(\alpha_js+\bar{\beta}_j)}-\log{\Gamma(1+\delta-s+\rho)}-\sum_{j=1}^{r}\log{\Gamma(\alpha_j(\delta-s)+\beta_j)}.
\end{equation} 
 For the remainder of the proof, $C_n \, (n=1,\dots, m)$ will denote constants whose values may differ from line to line. Using Lemma \ref{lem:gamma asymp}, we obtain
\begin{align}
	\log{G_\r(s)}&=\left(\d-s-\frac{1}{2}\right)\log{(-s)}+s+\sum_{j=1}^{r}\left\{\left(\alpha_js+\bar{\beta}_j-\frac{1}{2}\right)\log{\a_js} -\a_js \right\}\\
	&-\left(1+\delta-s+\rho-\frac{1}{2}\right)\log{(-s)}-s  -\sum_{j=1}^{r}\left\{\left(\alpha_j(\d-s)+\beta_j-\frac{1}{2}\right)\log{(-\a_js)}+\a_js \right\} \\
	&+\sum_{n=1}^{m}C_{n}s^{-n}+O\left(\frac{1}{|s|^{m+1}}\right) \label{eq:log G_rho step 1}\\
	&=\left(d_F^{'}s+2d_F^{'}\g-\mu\right)\log{(-s)}+\left(d_F^{'}s+\sum_{j=1}^{r}\bar{\beta}_j-\frac{r}{2}\right)\log s-2d_F^{'}s+2s\sum_{j=1}^{r}\a_j\log \a_j \\
	&-\sum_{j=1}^{r}\a_j \d \log \a_j-2i \sum_{j=1}^{r}c_j \log \a_j +
	 \sum_{n=1}^{m}C_{n} s^{-n}+O\left(\frac{1}{|s|^{m+1}}\right)\label{eq:fina log G_rho},
\end{align} where $\g, \mu$ are as defined in \eqref{notations in lem for I_rho}. For $v_0=1$ and some constants $v_n, n=1, \dots, m$, to be chosen later, we define 
\begin{align}\label{expression of F_n rho}
	F_n^{\r}(s)&:=\frac{v_nkh^s \G\left(2d_F^{'}s+2d_F^{'}\g-2i\nu-n\right)}{\G\left(\frac{1}{2}-d_F^{'}s-2d_F^{'}\g+\mu \right)\G\left(\frac{1}{2}+d_F^{'}s+2d_F^{'}\g-\mu \right)}\\
	&=\frac{v_nk}{\pi}h^s \G\left(2d_F^{'}s+2d_F^{'}\g-2i\nu-n\right)\cos\left(\pi(d_F^{'}s+2d_F^{'}\g-\mu)\right),\label{eq:expression of F_rho in cos}
\end{align}with $\g, \nu, \mu, h, k$ as defined in \eqref{notations in lem for I_rho}. Taking logarithm of $F_0^{\r}(s)$, we have
\begin{align}
	\log F_0^{\r}(s)=\log k +s\log h+\log \G\left(2d_F^{'}s+2d_F^{'}\g-2i\nu\right)&-\log \G\left(\frac{1}{2}-d_F^{'}s-2d_F^{'}\g+\mu \right) \\ &-\log \G\left(\frac{1}{2}+d_F^{'}s+2d_F^{'}\g-\mu \right).
\end{align} Using the values of $k,h$ and Lemma \ref{lem:gamma asymp}, we get
\begin{align}
	\log F_0^{\r}(s)&=\!  \left(\! -2d_F^{'}\g+2i\nu+\frac{1}{2}\right) \log 2d_F^{'}+\frac{\log{2\pi}}{2}-\sum_{j=1}^{r}\a_j \d \log \a_j-2i\sum_{j=1}^{r}c_j \log \a_j   \\
	&+2s\sum_{j=1}^{r}\a_j\log\a_j-2d_F^{'}s \log2d_F^{'}+\left(2d_F^{'}s+2d_F^{'}\g-2i\nu-\frac{1}{2}\right)\log 2d_F^{'}s-2d_F^{'}s	\\
	& +\frac{\log{2\pi}}{2}-\log{(-d_F^{'}s)}\left(-d_F^{'}s-2d_F^{'}\g+\mu \right)
	-d_F^{'}s-\frac{\log{2\pi}}{2}
	-\left(d_F^{'}s+2d_F^{'}\g-\mu \right)\log d_F^{'}s\\
	&+d_F^{'}s-\frac{\log{2\pi}}{2}+\sum_{n=1}^{m}C_ns^{-n}+\left(\frac{1}{|s|^{m+1}}\right)\\
&=\left(d_F^{'}s+2d_F^{'}\g-\mu\right)\log{(-s)}+\left(d_F^{'}s+\sum_{j=1}^{r}\bar{\beta}_j-\frac{r}{2}\right)\log s-2d_F^{'}s+2s\sum_{j=1}^{r}\a_j\log \a_j \\
	&-\sum_{j=1}^{r}\a_j \d \log \a_j-2i \sum_{j=1}^{r}c_j \log \a_j +
	\sum_{n=1}^{m}C_{n} s^{-n}+O\left(\frac{1}{|s|^{m+1}}\right)\label{eq:final log Frho},
\end{align} upon simplification. From \eqref{eq:fina log G_rho} and \eqref{eq:final log Frho}, we obtain 
\begin{equation}
	\log{G_\r(s)}=	\log F_0^{\r}(s)+\sum_{n=1}^{m}C_{n} s^{-n}+O\left(\frac{1}{|s|^{m+1}}\right).
\end{equation}
Thus, we have
\begin{align}
	G_\r(s)&=F_0^{\r}(s)\left[1+\sum_{n=1}^{m}C_{n} s^{-n}+O\left(\frac{1}{|s|^{m+1}}\right)\right]\\
	&=F_0^{\r}(s)\left[1+\sum_{n=1}^{m}\frac{C_{n}}{\left(2d_F^{'}s+2d_F^{'}\g-2i\nu-1\right)\cdots\left(2d_F^{'}s+2d_F^{'}\g-2i\nu-n\right)} +O\left(\frac{1}{|s|^{m+1}}\right)\right]\\
	&=F_0^{\r}(s)+\sum_{n=1}^{m}\frac{C_{n}kh^s \G\left(2d_F^{'}s+2d_F^{'}\g-2i\nu\right)\cos\left(\pi(d_F^{'}s+2d_F^{'}\g-\mu)\right)}
	{\pi \left(2d_F^{'}s+2d_F^{'}\g-2i\nu-1\right)\cdots\left(2d_F^{'}s+2d_F^{'}\g-2i\nu-n\right)} +O\left(\frac{|F_0^{\r}(s)|}{|s|^{m+1}}\right),
\end{align}upon substituting the value of $F_0^{\r}(s)$ from \eqref{eq:expression of F_rho in cos}. Since $z\G(z)=\G(z+1)$, putting $v_n=C_n$ and using \eqref{eq:expression of F_rho in cos}, we obtain
 \begin{align} \label{eq:expression of G_rho}
	G_\r(s)=\sum_{n=0}^{m}F^{\r}_n(s)+O\left(\frac{|F^{\r}_0(s)|}{|s|^{m+1}}\right).
\end{align}
We will integrate $G_\rho(s)$ along a contour $C_\r$ defined as follows. 
Let $c_\r=\frac{d_F^{'} \d+\r}{2d_F^{'}}-\epsilon, 0<\epsilon<\frac{1}{4d_F^{'}}$.
The contour $C_\r$ consists of the line $(c_\r-i\infty, c_\r-iT)$, followed by the three sides $[c_\r-iT, c_\r+u-iT], [c_\r+u-iT, c_\r+u+iT], [c_\r+u+iT, c_\r+iT]$ of a rectangle, followed by the line $(c_\r+iT, c_\r+i\infty)$.
Here we choose $u$ and $T$ sufficiently large, so that all the singularities of $F^{\r}_n(s)$ for $n=0, \cdots, m$, lie to the left of $C_\r$. 
If we take a contour $C$ which is the same as $C_\r$, with $c_\r$ replaced by some $c<-\frac{1}{2}$, and $u,T$ chosen so that all the singularities of $\G(s)\cos\left(\frac{\pi}{2}s+\a \right)x^{-s}$ lie to the left of $C$, then by Lemma \ref{lem:inv mel of Gammacos},
\begin{equation} \label{eq:intgammacos}
	\frac{1}{2\pi i}\int_{C}\G(s)\cos\left(\frac{\pi}{2}s+\a \right)x^{-s}ds=\cos(x+\a),
\end{equation}because of the contribution of residues at the poles $s=-n,0\leq n< |c|$, of the integrand.
Multiplying \eqref{eq:expression of G_rho} by $\frac{x^{\r+\d-s}}{2\pi i}$ and integrating  with respect to $s$ along the curve $C_\rho$, we have
\begin{equation} \label{eq:int of G_rho}
	\frac{x^{\r+\d}}{2\pi i}\int_{C_\r}G_\r(s) x^{-s}ds=\frac{x^{\r+\d}}{2\pi i} \left[\sum_{n=0}^{m}\int_{C_\r}F^{\r}_n(s)x^{-s}ds+\int_{C_\r}O\left(\frac{|F^{\r}_0(s)|}{|s|^{m+1}}\right)x^{-s}ds\right].
\end{equation}
 We will evaluate both the integrals on the right-hand side above. Using \eqref{eq:expression of F_rho in cos}, we have 
\begin{align}
	\int_{C_\rho}F^{\rho}_n(s)x^{-s}ds
	&=
	\frac{v_n k}{\pi}
	\int_{C_\rho}\Gamma \left(2d_F^{'}s+2d_F^{'}\gamma -2i \nu-n\right)
	\cos\left(\pi(d_F^{'}s+2d_F^{'}\gamma-\mu)\right)
	\left(h^{-1}x\right)^{-s}ds
	\\
	&=(h^{-1}x)^{\gamma-\frac{n+2i\nu}{2d_F^{'}}} \frac{v_nk}{2\pi d_F^{'}} \int_{C_\rho^{'}} \Gamma(z)
	\cos\left(\frac{\pi}{2}z+\pi\left(d_F^{'}\gamma+\frac{n}{2}+ i\nu- \mu\right)\right)
	(h^{-1}x)^{\frac{-z}{2d_F^{'}}}dz,
\end{align} 
upon changing the variable $2d_F^{'}s+2d_F^{'}\g-2i\nu-n$ to $z$.
The change of variable shifts the curve $C_\r$ to $C_\r^{'}$, with $c_\r$ replaced by $c{'}_\r<-\frac{1}{2}$.  Using \eqref{eq:intgammacos}, we obtain
\begin{equation} \label{eq:int of F_n}
	\frac{1}{2\pi i}
	\int_{C_\r} F^{\r}_n(s) x^{-s} ds
	=
	A_n
	x^{\g-\frac{n+2i\nu}{2d_F^{'}}}
	\cos\left((h^{-1}x)^{\frac{1}{2d_F^{'}}}+\pi\left(d_F^{'}\g + \frac{n}{2}+ i\nu- \mu\right)\right),
\end{equation} where $A_n=\frac{v_nk}{2\pi d_F^{'}}h^{\frac{n+2i\nu}{2d_F^{'}}-\g}$. 
In order to deal with the second integral in \eqref{eq:int of G_rho}, we write
\begin{equation}\label{eq:error term F_rho}
	\frac{1}{2\pi i}\int_{C_\r}O\left(\frac{|F^{\r}_0(s)|}{|s|^{m+1}}\right)x^{-s}ds=\frac{1}{2\pi i}\int_{C_\r+\frac{m+1}{2d_F^{'}}}O\left(\frac{|F^{\r}_0(s)|}{|s|^{m+1}}\right)x^{-s}ds,
\end{equation}
where $C_\r+\frac{m+1}{2d_F^{'}}$ denotes the contour obtained by shifting $C_\r$ to the right by $\frac{m+1}{2d_F^{'}}$. This shift is justified as follows. Since all the singularities of the integrand lie to the left of $C_\r$, it is analytic  between the contours $C_\r$ and $C_\r+\frac{m+1}{2d_F^{'}}$.  
Using \eqref{expression of F_n rho}, the Stirling formula \eqref{eq:stirling} and the fact that $\cos(s)\ll \exp(|\Im(s)|)$, writing $s=\s+it$, in the region $c_\r \leq \s \leq c_\r+\frac{m+1}{2d_F^{'}}$, we see as $|t| \to \infty$,
	\begin{equation}\label{eq:estimation of F_0 upon s}
	\frac{|F^{\r}_0(s)| }{|s|^{m+1}}\ll 
	\frac{1}{|s|^{m+1}} |2d_F^{'}t-2\nu|^{2d_F^{'}\left(c_\r+\tfrac{m+1}{2d_F^{'}}\right)+2d_F^{'}\g-\frac{1}{2}}
	\ll |t|^{-1-2\epsilon d_F^{'}},
\end{equation} 
upon putting 
$c_\r=\frac{d_F^{'} \d+\r}{2d_F^{'}}-\epsilon, \g=-\left(\frac{\d}{2}+\frac{\r}{2d_F^{'}}+\frac{1}{4d_F^{'}}\right)$ 
and  using $|s|\sim |t|$. Hence, by analyticity and decay of the integrand which makes the contribution of the horizontal integrals zero as $|t|\to \infty$, \eqref{eq:error term F_rho} is justified. 

For any large fixed $\epsilon>0$,  the right-hand side of \eqref{eq:error term F_rho} is 
\begin{equation}\label{eq:int of F_0}
	\ll x^{-c_\r-\frac{m+1}{2d_F^{'}}}
	\left( \int_{|t|> \epsilon }\frac{|F^{\r}_0(s)|}{|s|^{m+1}}dt  + O_\epsilon(1) \right)
	\ll x^{-c_\r-\frac{m+1}{2d_F^{'}}}, 
\end{equation}
using  \eqref{eq:estimation of F_0 upon s}. 
Combining \eqref{eq:int of G_rho}, \eqref{eq:int of F_n}, \eqref{eq:int of F_0}, recalling that $c_\r=\frac{d_F^{'} \d+\r}{2d_F^{'}}-\epsilon$, 
and putting 
$\omega'=d_F^{'}\d+(2d_F^{'}-1)\r$, 
we have
\begin{align}
	&\frac{x^{\rho+\delta}}{2\pi i}
	\int_{C_\rho} 
	G_\rho(s) x^{-s}ds
	\\
	&=
	\sum_{n=0}^{m} A_n x^{\frac{\omega'-n-2i\nu-\frac{1}{2}}{2d_F^{'}}}
	\cos\left((h^{-1}x)^{\frac{1}{2d_F^{'}}}+\pi\left(d_F^{'}\gamma+ \frac{n}{2}+ i\nu- \mu\right)\right) 
	+ O\bigg(x^{\frac{\omega'-m-1}{2d_F^{'}}+\epsilon}\bigg)\\
	\\
	&=
	\sum_{n=0}^{m-1} A_n x^{\frac{\omega'-n-2i\nu-\frac{1}{2}}{2d_F^{'}}}
 \cos\left((h^{-1}x)^{\frac{1}{2d_F^{'}}}+\pi\left(d_F^{'}\gamma+\frac{n}{2}+ i\nu- \mu\right)\right)
	 + O\bigg( x^{\frac{\omega'-m-\frac{1}{2}}{2d_F^{'}}}+x^{\frac{\omega'-m-1}{2d_F^{'}}+\epsilon}\bigg).
	 \label{eq:finalint G_rho}
\end{align}
Since 
$\frac{\omega'-m-1}{2d_F^{'}}+\epsilon < \frac{\omega'-m-\frac{1}{2}}{2d_F^{'}}$ 
for 
$0<\epsilon<\frac{1}{4d_F^{'}}$,  
we obtain 
\begin{equation}\label{eq:final G_rho}
	\frac{x^{\rho+\delta}}{2\pi i}\int_{C_\rho}G_\rho(s) x^{-s}ds=\sum_{n=0}^{m-1} A_n x^{\frac{\omega'-n-2i\nu-\frac{1}{2}}{2d_F^{'}}}
 \cos\left((h^{-1}x)^{\frac{1}{2d_F^{'}}}+\pi\left(d_F^{'}\gamma+\frac{n}{2}+ i\nu- \mu\right)\right) 
	+ O\left( x^{\frac{\omega'-m-\frac{1}{2}}{2d_F^{'}}}\right).
\end{equation}
 To complete the proof, we need to show that
\begin{equation} \label{eq:G_rho equality}
	\int_{C_\r}G_\r(s) x^{-s}ds=\int_{a-i\infty}^{a+i \infty}G_\r(s) x^{-s}ds,
\end{equation} for $a$ sufficiently large. Note that \eqref{eq:final G_rho} holds for the contour $C_\r$, with any $u,T$ sufficiently large. Denoting $c_\r+u$ as $a$, by definition of the contour $C_\r$, we have 
\begin{align}\label{eq:int of G_rho on C}
	\int_{C_\r}G_\r(s) x^{-s}ds&=\left[\int_{c_\r-i\infty}^{c_\r-iT}+\int_{c_\r-iT}^{a-iT}+\int_{a-iT}^{a+iT}+\int_{a+iT}^{c_\r+iT}+\int_{c_\r+iT}^{c_\r+i\infty}\right]G_\r(s) x^{-s}ds.
\end{align}
We now claim that the horizontal integrals given by the second and fourth terms on the right-hand side above, vanish as $T\rightarrow \infty$. We use \eqref{def of G_rho} and the Stirling formula \eqref{eq:stirling} to obtain after some simplification,
\begin{equation}
	G_\r(s) 	\ll |T|^{-1-\r+(2a-\d)d_F^{'}},
\end{equation} which is $o(1)$  as $T\rightarrow \infty$, whenever $\r>(2a-\d)d_F^{'}-1$. This justifies the vanishing of the above-mentioned horizontal integrals in \eqref{eq:int of G_rho on C}. As $T\to \infty$, the remaining integrals in \eqref{eq:int of G_rho on C} yield the right-hand side of \eqref{eq:G_rho equality}, as needed.
\end{proof}

	\section{Proof of Theorem \ref{thm:FE implies mod rel1}}

		Let us start with the Mellin inversion of $F(s)$. We choose $a>\max{\{0,\s_a,\s_b\}}$ and $T$ large 
		 enough such that the rectangle $R$ with vertices $a\pm iT, \d-a \pm iT$ encloses all the singularities of $F(s)$. From the definition of $F(s)$ and from \eqref{def:dirser}, we have
		\begin{align} \label{mellin inv of F}
		 \frac{1}{2 \pi i}\int_{a-i\infty}^{a+i\infty}F(s) x^{-s}ds
			&=\lim_{T\rightarrow \infty} \frac{1}{2 \pi i}\int_{a-iT}^{a+iT}\prod_{i=1}^{r}\Gamma(\alpha_is+\beta_i)\phi(s)x^{-s}ds \\
			&=\frac{1}{2\pi i} \int_{a-i\infty}^{a+i\infty} \sum_{n=1}^{\infty}a_n  \prod_{i=1}^{r}\Gamma(\alpha_is+\beta_i)(\lambda_nx)^{-s}ds. \label{eq:lhs of mod rel}
		\end{align}
By Lemma \ref{lem:exp order of Z_aplha,beta}, we have 
	\begin{align} 
		\sum_{n=1}^{\infty}a_n\frac{1}{2\pi i} \int_{a-i\infty}^{a+i\infty} \prod_{i=1}^{r}\Gamma(\alpha_is+\beta_i)(\lambda_nx)^{-s}ds
		&\ll \sum_{n=1}^{\infty}|a_n| \exp{\Big(-c(\lambda_nx)^{\frac{1}{d_F^{'}}}\Big)} \\
		&\ll 	\sum_{n=1}^{\infty}|a_n| (\lambda_nx)^{-a},\label{abs conv of sum_int}
	\end{align} which is absolutely convergent since $a>\s_a$, where $\s_a$ is the abscissa of absolute convergence of $\phi(s)$. This justifies the interchange of summation and integration in \eqref{eq:lhs of mod rel}, to yield
		\begin{align} \label{LHS}
		 \frac{1}{2 \pi i}\int_{a-i\infty}^{a+i\infty}F(s) x^{-s}ds	=\sum_{n=1}^{\infty}a_n\frac{1}{2\pi i} \int_{a-i\infty}^{a+i\infty} \prod_{i=1}^{r}\Gamma(\alpha_is+\beta_i)(\lambda_nx)^{-s}ds.	
		\end{align}
		
We will now obtain a different expression for the left-hand side of \eqref{LHS} in order to derive the modular relation. Consider a curve $C$ lying inside the rectangle $R$ such that $C$ encloses all the singularities of $F(s)$. As all the singularities of $F(s)x^{-s}$ lie inside $C$, by Cauchy's residue theorem, we get
		\begin{multline}
			\frac{1}{2 \pi i}\int_{a-iT}^{a+iT}F(s) x^{-s}ds = \frac{1}{2 \pi i}\int_{C} F(s) x^{-s}ds-\frac{1}{2 \pi i}\left[\int_{a+iT}^{\delta-a+iT}+\int_{\delta-a+iT}^{\delta-a-iT}+\int_{\delta-a-iT}^{a-iT}\right]F(s) x^{-s}ds. \label{eq:contourintthm2.1}
		\end{multline} 
	We will show that the integrals on the horizontal lines will tend to zero as $T \rightarrow \infty$. Using Lemma \ref{lem:phragmen-lind}, we first show that
	\begin{equation}
		F(s)x^{-s}=o(1),
	\end{equation}
uniformly in the strip $\d-a \leq \Re(s)\leq a$ as $|\Im(s)|\rightarrow \infty$. From Definition \ref{def:fn eqn of F}, we see that $F(s)x^{-s}=\chi(s)(Qx)^{-s}$ satisfies the  big $O$ hypothesis required for Lemma \ref{lem:phragmen-lind}. Now it remains to verify the growth condition for $F(s)x^{-s}$ on the two vertical lines, $\Re(s)=a$ and $\Re(s)=\d-a$.

 On the line $\Re(s)=a$, we have $\phi(s)x^{-s}=O(1)$ as $\phi(s)$ is absolutely convergent on this line. Writing $s=a+it, \b_i=b_i+ic_i$, and using the asymptotic formula \eqref{eq:stirling}, we have
\begin{equation}
	\left|\prod_{i=1}^{r}\Gamma(\a_i s+\b_i)\right| \sim e^{-\frac{\pi}{2} \sum_{i=1}^{r}  |\a_i t+c_i|} \prod_{i=1}^{r}|\a_i t+c_i|^{\a_i a+b_i-\frac{1}{2}}(\sqrt{2\pi})^r,
\end{equation}
which implies that $\prod_{i=1}^{r}\Gamma(\a_i s+\b_i)=o(1)$ as $|t|\rightarrow \infty$. Thus, $F(s)x^{-s}=o(1)$ on the line $\Re(s)=a$.

Similarly, on the line $\Re(s)=\d-a$, we substitute the functional equation \eqref{fn eqn} and express $F(s)x^{-s}$ by
\begin{equation}
F(s)x^{-s}= \omega x^{-s} Q^{\d-2s} \overline{G(\d-\bar{s})}.
\end{equation}
Again, using the absolute convergence of $\psi(s)$ on the line $\Re(s)=\d-a$ and the Stirling formula for the Gamma factors, we conclude that $F(s)x^{-s}=o(1)$ on $\Re(s)=\d-a$. Applying Lemma \ref{lem:phragmen-lind}, we have $F(s)x^{-s}=o(1)$, uniformly in the strip $\d-a\leq \Re(s)\leq a$ as $|t|\rightarrow \infty$ and hence both the horizontal integrals vanish as $|T|\rightarrow \infty$. Thus, from \eqref{eq:contourintthm2.1}, we have
		\begin{align}
			\frac{1}{2 \pi i}\int_{a-i\infty}^{a+i\infty}F(s) x^{-s}ds&=\frac{1}{2 \pi i}\int_{C} F(s) x^{-s}ds+\frac{1}{2 \pi i}\int_{\delta-a-i\infty}^{\delta-a+i\infty}F(s) x^{-s}ds	\\
			&=P(x)+\frac{1}{2 \pi i}\int_{\delta-a-i\infty}^{\delta-a+i\infty}w Q^{\delta-2s} \overline{G(\delta-\bar{s})}x^{-s}ds,
		\end{align}
		upon using the functional equation \eqref{fn eqn} and denoting    the integral $\frac{1}{2 \pi i}\int_{C} F(s) x^{-s}ds$ by $P(x)$.	
		Now, we change the variable $\delta-s$ to $s$, to get
		\begin{align}
			\frac{1}{2 \pi i}\int_{a-i\infty}^{a+i\infty}F(s) x^{-s}ds
			&=P(x)+\frac{1}{2 \pi i}\int_{a-i\infty}^{a+i\infty}w Q^{2s-\delta} \overline{G(\bar{s})}x^{-\delta+s}ds\\
			&=P(x)+ \frac{\omega}{(xQ)^\delta}\sum_{n=1}^{\infty}\bar{b}(n)\frac{1}{2 \pi i}\int_{a-i\infty}^{a+i\infty} \prod_{i=1}^{r}\Gamma(\alpha_is+\bar{\beta_i})\bigg(\frac{\mu_n}{Q^2 x}\bigg)^{-s}ds,\label{RHS}
		\end{align}
	after interchanging the summation and integration, and using the definition of $G(s)$. Again, the interchange can be justified by absolute convergence of the above sum as explained in \eqref{abs conv of sum_int}. Combining \eqref{LHS} and \eqref{RHS}, we get the modular relation \eqref{mod relation1}.
	
	\section{Proof of Theorem \ref{thm:mod rel1 implies FE}}

To obtain the functional equation \eqref{fn eqn} from the modular relation \eqref{mod relation1}, we start with the expression for $Z_{\bs \alpha,\bs\beta}(\lambda_nx)$. From \eqref{def:Z_alpha,beta}, for  $x,a>0$, we have	
	\begin{equation}
		Z_{\bs \alpha,\bs\beta}(\lambda_nx)=\frac{1}{2 \pi i}\int_{a-i\infty}^{a+i\infty} \prod_{i=1}^{r} \Gamma(\alpha_i s+ \beta_i) (\lambda_nx)^{-s}ds.
	\end{equation}
	 Taking the Mellin transform, for $\Re(s)=\s>\max{\{0, \s_a, \s_b\}}$, we write 
	\begin{align}
		\prod_{i=1}^{r} \Gamma(\alpha_i s+ \beta_i)&= \int_{0}^{\infty} Z_{\bs \alpha,\bs\beta}(\lambda_nx)(\lambda_nx)^{s-1} d(\lambda_nx)\\
		&=\lambda_n^{s}\int_{0}^{\infty} Z_{\bs \alpha,\bs\beta}(\lambda_nx)x^{s-1} dx.
	\end{align}
	We multiply both sides by $Q^s a_{n}$ and then sum over $n$ to get left-hand side of the functional equation \eqref{fn eqn}:
	\begin{equation}\label{sum and int 3}
		Q^s \phi(s) \prod_{i=1}^{r} \Gamma(\alpha_i s+ \beta_i)=Q^s \sum_{n=1}^{\infty} a_n \int_{0}^{\infty} Z_{\bs \alpha,\bs\beta}(\lambda_nx)x^{s-1} dx.
	\end{equation}

In order to interchange summation and integration in the above expression on the right-hand side, let us first check the absolute convergence of the same. Using Lemma \ref{lem:exp order of Z_aplha,beta}, we can write
\begin{align}
	\sum_{n=1}^{\infty} a_n \int_{0}^{\infty} Z_{\bs \alpha,\bs\beta}(\lambda_nx)x^{s-1} dx&\ll \sum_{n=1}^{\infty} |a_n| \int_{0}^{\infty} \exp{\big(-c(\lambda_nx)^{\frac{1}{d_F^{'}}}\big)}x^{\s-1} dx\\
	&\ll \sum_{n=1}^{\infty} |a_n| \lambda_n^{-\s} \int_{0}^{\infty} \exp{\big(-cy^{\frac{1}{d_F^{'}}}\big)}y^{\s-1} dy, \label{interchangejusti3}
\end{align}
after changing the variable $\lambda_nx$ to $y$. Here, the sum $\sum_{n=1}^{\infty}a_n \lambda_n^{-\s}$ converges  absolutely as $\s> \s_a$ and the integral above converges due to the exponential decay of the integrand. Hence, we may interchange the order of summation and integration in \eqref{sum and int 3}, to obtain
	\begin{equation}
			Q^s F(s)= Q^s\bigg( \int_{0}^{1} \sum_{n=1}^{\infty} a_n Z_{\bs \alpha,\bs\beta}(\lambda_nx)x^{s-1} dx+\int_{1}^{\infty} \sum_{n=1}^{\infty} a_n Z_{\bs \alpha,\bs\beta}(\lambda_nx)x^{s-1} dx\bigg).
	\end{equation}
Now, we apply the modular relation \eqref{mod relation1} to the first integral to get
	\begin{align}
	Q^s F(s)	&= Q^s\bigg( \int_{0}^{1} P(x)x^{s-1} dx+\int_{0}^{1}  \Psi(1/x)x^{s-1} dx+\int_{1}^{\infty} \Phi(x)x^{s-1} dx\bigg)\\
		&=Q^s\bigg( \int_{0}^{1} P(x)x^{s-1} dx+\int_{1}^{\infty}  (\Psi(x)x^{-s-1} + \Phi(x)x^{s-1}) dx\bigg)\\
		&=I_1+J_1 \quad \text{(say)}, \label{lhs of fn eq}
	\end{align} 
	where 
	\begin{equation}
		\Phi(x)=\sum_{n=1}^{\infty} a_n Z_{\bs \alpha,\bs\beta}(\lambda_nx) \hspace{1cm} \text{and} \hspace{1cm} \Psi(x)= \frac{\omega x^\delta}{Q^\delta} \sum_{n=1}^{\infty} \bar{b}_n Z_{\bs\alpha,\bs{\bar{\beta}}}\left(\frac{\mu_nx}{Q^2}\right).
	\end{equation}
This gives us an expression for the left-hand side of the desired functional equation \eqref{fn eqn} in terms of integrals involving $Z_{\bs \alpha,\bs\beta}(x)$. We now attempt to do this for the right-hand side of \eqref{fn eqn} in a similar manner. From \eqref{def:Z_alpha,beta}, for  $x,a>0$, we have 
	\begin{equation}
		Z_{\bs\alpha,\bs{\bar{\beta}}}\bigg(\frac{\mu_nx}{Q^2}\bigg)=\frac{1}{2 \pi i}\int_{a-i\infty}^{a+i\infty} \prod_{i=1}^{r} \Gamma(\alpha_i s+ \bar{\beta_i}) \bigg(\frac{\mu_nx}{Q^2}\bigg)^{-s}ds.
	\end{equation}
	Again, taking the Mellin transform, multiplying both sides by $\omega Q^s\bar{b}_n$, and summing over $n$, we get for $\Re(s)>\max{\{0, \s_a, \s_b\}}$,
	\begin{align}
		\omega Q^s  \overline{G(\bar{s})} 
		&= \omega Q^{-s} \sum_{n=1}^{\infty} \bar{b}_n \int_{0}^{\infty} Z_{\bs\alpha,\bs{\bar{\beta}}}\bigg(\frac{\mu_nx}{Q^2}\bigg)x^{s-1} dx \\
			&= \omega Q^{-s}  \left[\int_{0}^{1}+ \int_{1}^{\infty}\right]\sum_{n=1}^{\infty} \bar{b}_n Z_{\bs\alpha,\bs{\bar{\beta}}}\bigg(\frac{\mu_nx}{Q^2}\bigg)x^{s-1} dx.
	\end{align} 
	Here the interchange is justified by using Lemma \ref{lem:exp order of Z_aplha,beta} as demonstrated before in \eqref{interchangejusti3}. Proceeding in a similar manner, we use the modular relation \eqref{mod relation1} in the first integral above, to obtain
	\begin{align}
		\omega Q^s \overline{G(\bar{s})}&= \omega Q^{-s}\bigg( \frac{Q^\delta}{\omega}\int_{0}^{1} \frac{1}{x^\delta}\Big(\Phi(1/x)-P(1/x)\Big) x^{s-1}+ \frac{Q^\delta}{\omega}\int_{1}^{\infty}\Psi(x)x^{-\delta+s-1} dx\bigg)\\
		&=Q^{\delta-s}\left(-\int_{1}^{\infty}P(x)x^{\delta-s-1}dx+\int_{1}^{\infty}(\Psi(x)x^{-\delta+s-1}+\Phi(x)x^{\delta-s-1})dx\right) \label{rhs of fn eqn}\\
		&= I_2+J_2 \quad \text{(say)}. \label{final rhs of fn eq}
	\end{align}
	On replacing $s$ by $\delta-s$ in \eqref{rhs of fn eqn}, we notice that $J_1$ and $J_2$ in \eqref{lhs of fn eq} and \eqref{final rhs of fn eq} respectively become equal, and by Definition \ref{def:res fn}, $I_1$ and $I_2$ are equal upon analytic continuation into the domain $D$. Hence the functional equation \eqref{fn eqn} follows.

\section{Proof of Theorem \ref{thm:FE implies Rieszsum}}

               
	We know that $\phi(s)=\sum_{n=1}^{\infty}a_n \lambda_n^{-s}$ and $\psi(s)=\sum_{n=1}^{\infty}b_n \mu_n^{-s}$ converge absolutely for $\Re(s)>\s_a$ and $\Re(s)>\s_b$ respectively. Now take $a=\s_b+\frac{1}{4d_F^{'}}+\frac{m}{2d_F^{'}}$, where $m$ is a sufficiently large integer such that $a>\max {\{0, \s_a, \s_b\}}$ and the set $S$ (as in Definition \ref{def:fn eqn of F}) lies inside the strip $\delta-a<\Re(s)<a$. Using Lemma \ref{lem:perronsformula}, for $\rho\geq0$, we can write
	\begin{equation}
		\frac{1}{\Gamma(\rho+1)}\sideset{}{'}\sum_{\lambda_n \leqslant x}{} a_{n}(x-\lambda_n)^\rho=\frac{1}{2\pi i} \int_{a-i \infty}^{a+i \infty} \frac{\phi(s)\Gamma(s) x^{s+\rho}}{\Gamma(s+\rho+1)}ds	. \label{lhs of riez sum}
	\end{equation} 
	Let us define $f(s):=\frac{\phi(s)\Gamma(s) x^{s+\rho}}{\Gamma(s+\rho+1)}$. We choose $T$ to be large enough such that $S$ lies inside the rectangle $R$ with vertices $a \pm iT$, $\delta-a\pm iT$. Subsequently, take a curve $C$ inside $R$, such that no singularities of $f(s)$ lie in the region between the rectangle $R$ and the curve $C$. Using the Cauchy residue theorem, we have
	\begin{equation}\label{integral of f(s)}
		\frac{1}{2 \pi i}\int_{a-iT}^{a+iT}f(s) ds = \frac{1}{2 \pi i}\int_{C} f(s) ds-\frac{1}{2 \pi i}\bigg[\int_{a+iT}^{\delta-a+iT}+\int_{\delta-a+iT}^{\delta-a-iT}+\int_{\delta-a-iT}^{a-iT}\bigg]f(s)ds.
	\end{equation}
	In order to show that the horizontal integrals are negligible as $T\rightarrow \infty$, it is enough to show that $f(s)=o(1)$ 
	uniformly in the strip $\d-a \leq \Re(s) \leq a$. We will show this by applying Lemma \ref{lem:phragmen-lind} to the function $f(s)$.
	
	On the line $\Re(s)=a,\phi(s)$ converges absolutely and so we have $\phi(s)x^{\r+a}\ll 1$. Using Stirling's formula \eqref{eq:stirling}, we find that
\begin{equation} \label{approx of two gamma frac}
	\frac{\G(s)}{\G(s+\r+1)} \sim |t|^{-\r-1},
\end{equation}
which is $o(1)$ since $\r> 0$. Hence we have $f(s)=o(1)$ on the line $\Re(s)=a$.

On the other vertical line $\Re(s)=\d-a$, we use the functional equation \eqref{fn eqn} and get
\begin{equation}
	f(s)=	\frac{\phi(s)\Gamma(s) x^{s+\rho}}{\Gamma(s+\rho+1)}=\frac{\G(s)\prod_{i=1}^{r}\G(\a_i(\d-s)+\bar{\b}_i)}{\G(s+\r+1)\prod_{i=1}^{r}\G(\a_is+\b_i)}\omega Q^{\d-2s}\overline{\psi(\d-\bar{s})}.
\end{equation}
Again, $\omega Q^{\d-2s}\overline{\psi(\d-\bar{s})}\ll 1$ on $\Re(s)=\d-a$ since $\psi(s)$ converges absolutely for $\Re(s)>\s_b$. Writing $\b_i=b_i+ic_i$ and $s=\d-a+it$, we apply  Stirling's formula \eqref{eq:stirling} to get after some computations,
\begin{equation}
	\frac{\G(s)\prod_{i=1}^{r}\G(\a_i(\d-s)+\bar{\b}_i)}{\G(s+\r+1)\prod_{i=1}^{r}\G(\a_is+\b_i)}\sim \frac{|t|^{\d-a}  \prod_{i=1}^{r}|\a_i t+c_i|^{\a_i a+b_i}}{|t|^{\d-a+\r+1}\prod_{i=1}^{r}|\a_i t+c_i|^{\a_i \d-\a_i a+b_i}},
\end{equation} as $|t|\rightarrow \infty$. On simplifying, this gives 
\begin{equation}
	f(s)\ll_{\bs \a,\bs \b} |t|^{-\r-1+(2a-\d) \sum_{i=1}^{r}\a_i},
\end{equation} which is $o(1)$ provided $\r>(2a-\d)d_F^{'}-1$.

Finally, in order to apply Lemma \ref{lem:phragmen-lind}, we need to bound $f(s)$ inside the strip $\d-a \leq \Re(s) \leq a$. Inside this strip, we may write $f(s)$ as 
\begin{equation}
f(s)=\frac{\Gamma(s) \chi(s) x^{s+\rho}}{\Gamma(s+\rho+1) \prod_{i=1}^r \Gamma\left(\alpha_i s+\beta_i \right) Q^s},
\end{equation} where $\chi(s)$ is as in Definition \ref{def:fn eqn of F}.
 As $\frac{1}{\G(s)}$ is an analytic function of order $1$, we have 
\begin{equation}
	\frac{1}{ \prod_{i=1}^r \Gamma\left(\alpha_i s+\beta_i \right)}= O\left(\exp\left(r|s|\right)\right).
\end{equation} Combining this with \eqref{chigrowthcond} and \eqref{approx of two gamma frac}, we have
\begin{equation}
	f(s)\ll \exp\left(\exp\left(\frac{\theta \pi|s|}{2a-\d}\right)\right),
\end{equation} uniformly for $\d-a \leq \Re(s) \leq a$ as needed. We may now apply Lemma \ref{lem:phragmen-lind} to conclude that the horizontal integrals in \eqref{integral of f(s)} vanish as $T \rightarrow \infty$, provided $\r> \max{\{0, (2a-\d)d_F^{'}-1\}}$.

Defining $Q_\r(x)$ as in \eqref{resfnQ_rho}, we have thus obtained from \eqref{integral of f(s)},
	\begin{align}
		\frac{1}{2 \pi i}\int_{a-i\infty}^{a+i\infty}f(s)ds&=\frac{1}{2 \pi i}\int_{C} f(s) ds+ 	\frac{1}{2 \pi i}\int_{\delta-a-i\infty}^{\delta-a+i\infty} \frac{\phi(s)\Gamma(s) x^{s+\rho}}{\Gamma(s+\rho+1)}ds \\
		&=Q_\rho(x)+\frac{1}{2 \pi i}\int_{\delta-a-i\infty}^{\delta-a+i\infty} \frac{\Gamma(s) x^{s+\rho}\omega Q^{\delta-2s}\prod_{i=1}^{r}\Gamma(\alpha_i(\delta-s)+\bar{\beta}_i)\overline{\psi(\delta-\bar{s})}}{\Gamma(s+\rho+1)\prod_{i=1}^{r}\Gamma(\alpha_is+\beta_i)}ds,
	\end{align}
 upon using the functional equation \eqref{fn eqn} and denoting the integral $\frac{1}{2 \pi i}\int_{C} f(s) ds$ by $Q_\rho(x)$. Changing the variable $\delta-s$ to $ s$ on the right-hand side above, we get
 \begin{equation}
\frac{1}{2 \pi i}\int_{a-i\infty}^{a+i\infty}f(s)ds =   
Q_\rho(x)+\frac{1}{2 \pi i}\int_{a-i\infty}^{a+i\infty} \frac{\Gamma(\delta-s) x^{\delta-s+\rho}\omega Q^{2s-\delta}\prod_{i=1}^{r}\Gamma(\alpha_is+\bar{\beta}_i)\overline{\psi(\bar{s})}}{\Gamma(1+\delta-s+\rho)\prod_{i=1}^{r}\Gamma(\alpha_i(\delta-s)+\beta_i)}ds.\label{rhsintegral}
 \end{equation}
We will now justify the interchange of the integral and sum (coming from $\overline{\psi(\bar{s})}=\sum_{n=1}^{\infty}\bar{b}_n\mu_n^{-s}$) on right-hand side of the above expression as follows.

Let us define 
\begin{equation}
	I_\r\left(\frac{\mu_nx}{Q^2}\right):=\frac{1}{2 \pi i}\int_{a-i\infty}^{a+i\infty} \frac{\Gamma(\delta-s) \prod_{i=1}^{r}\Gamma(\alpha_is+\bar{\beta}_i)}{\Gamma(1+\delta-s+\rho)\prod_{i=1}^{r}\Gamma(\alpha_i(\delta-s)+\beta_i)}\left(\frac{\mu_nx}{Q^2}\right)^{\d+\r-s}ds
\end{equation} and 
	\begin{equation}
		J_\r(x):=\sum_{n=1}^{\infty}\bar{b}_n  \frac{Q^{\d+2\r}}{\mu_n^{\d+\r}}I_\r\left(\frac{\mu_nx}{Q^2}\right).
	\end{equation}
	From Lemma \ref{lem:asym for I_rho}, we deduce as $x\rightarrow \infty$,
	\begin{equation}
		I_\rho \left(\frac{\mu_nx}{Q^2}\right)\sim  \left(\frac{\mu_nx}{Q^2}\right)^{\frac{d_F^{'}\delta+(2d_F^{'}-1)\rho-2i\nu-\frac{1}{2}}{2d_F^{'}}}\sum_{q=0}^{m} A_q \left(\frac{\mu_nx}{Q^2}\right)^{-\frac{q}{2d_F^{'}}}
	 \cos\left(\left(\frac{\mu_nx}{hQ^2}\right)^\frac{1}{2d_F^{'}}+\pi\left(d_F^{'}\gamma+\frac{q}{2}+ i\nu- \mu\right)\right),
	\end{equation}
 where $A_q$ are some constants and $\g, \nu, \mu, h$ are as defined in \eqref{notations in lem for I_rho}. 
  Using this, we have 
	\begin{equation} \label{asymp of J_rho}
		J_\r(x)\ll \sum_{n=1}^{\infty}|b_n|  \frac{Q^{\d+2\r}}{|\mu_n|^{\d+\r}} \left(\frac{|\mu_n|x}{Q^2}\right)^{\frac{d_F^{'}\d+(2d_F^{'}-1)\r-\frac{1}{2}}{2d_F^{'}}}.
	\end{equation}
	As $\sum_{n=1}^{\infty}\frac{b_n}{\mu_n^{s}}$ converges absolutely  for $\Re(s)> \s_b$, from \eqref{asymp of J_rho} we deduce that $J_\r(x)$  converges absolutely and uniformly in a compact domain for $\r> 2d_F^{'}\s_b-d_F^{'}\d-\frac{1}{2}$. We thus interchange the order of summation and integration in \eqref{rhsintegral}. Finally combining this with \eqref{lhs of riez sum} yields
\begin{align}
		&\frac{1}{\Gamma(\rho+1)}\sideset{}{'}\sum_{\lambda_n \leqslant x}{} a_{n}(x-\lambda_n)^\rho\\
		&=Q_\rho(x)+\frac{\omega x^{\delta+\rho}}{Q^\delta} \sum_{n=1}^{\infty}\bar{b}_n\frac{1}{2 \pi i}\int_{a-i\infty}^{a+i\infty} \frac{\Gamma(\delta-s) \prod_{i=1}^{r}\Gamma(\alpha_is+\bar{\beta}_i)}{\Gamma(1+\delta-s+\rho)\prod_{i=1}^{r}\Gamma(\alpha_i(\delta-s)+\beta_i)}\bigg(\frac{\mu_nx}{Q^2}\bigg)^{-s}ds,	\label{rhs of riez sum}
\end{align}
for $\r \geq 0$ and $\r >(2a-\d)d_F^{'} -1$. 
However, we can differentiate the identity \eqref{rhs of riez sum} with respect to $x$ and extend the validity of the identity for smaller values of $\r$, since differentiating with respect to $x$ gives
\begin{multline}
	\frac{1}{\Gamma(\rho)}\sideset{}{'}\sum_{\lambda_n \leqslant x}{} a_{n}(x-\lambda_n)^{\rho-1}=\\Q_{\rho-1}(x)+\frac{\omega x^{\delta+\rho-1}}{Q^\delta} \sum_{n=1}^{\infty}\bar{b}_n\frac{1}{2 \pi i}\int_{a-i\infty}^{a+i\infty} \frac{\Gamma(\delta-s) \prod_{i=1}^{r}\Gamma(\alpha_is+\bar{\beta_i)}}{\Gamma(\delta-s+\rho)\prod_{i=1}^{r}\Gamma(\alpha_i(\delta-s)+\beta_i)}\bigg(\frac{\mu_nx}{Q^2}\bigg)^{-s}ds,\label{diff rhs of riesz sum}
\end{multline} which is simply the identity \eqref{rhs of riez sum} with $\r$ replaced by $\r-1$. In order to justify this differentiation and to deduct how many times one can apply it repeatedly, we first observe that the series on the right-hand side of \eqref{rhs of riez sum} is uniformly convergent for $\r >(2\s_b-\d)d_F^{'} -\frac{1}{2}$. Hence, for any $\r >(2\s_b-\d)d_F^{'} -\frac{1}{2}$, the right-hand side of \eqref{rhs of riez sum} is continuously differentiable (since $Q_\r(x)$ is $C^\infty$ from Lemma \ref{lem:Q_rho is C infinite}). In particular, for any such value of $\r$, the right-hand side of \eqref{rhs of riez sum} is continuous. Since the values of $a_n$ on the left-hand side of \eqref{rhs of riez sum} are not all zero, for the left-hand sum to be continuous, we must have the exponent $\r>0$. We thus conclude that $(2\s_b-\d)d_F^{'}-\frac{1}{2}$ must be non-negative. Since $(2a-\d)d_F^{'}-1=(2\s_b-\d)d_F^{'}-\frac{1}{2}+m\geq m$, this means that \eqref{rhs of riez sum} can be differentiated $m$ times to yield the same identity for any $\r>(2\s_b-\d)d_F^{'}-\frac{1}{2}$. This completes the proof.



\section{Proof of Theorem \ref{thm:riesz sum implies modreln 2}}

 We first multiply the Riesz sum identity \eqref{riesz sum} throughout by 
\begin{equation}
	A(x):=y^{\r+1} \frac{1}{2\pi i}	\int_{k-i\infty}^{k+i\infty} \Gamma(s)\prod_{i=1}^{r}\Gamma\left(\alpha_i(s-\r-1)+\beta_i\right) (xy)^{-s}ds,
\end{equation} for a suitable $k>0$
and then integrate with respect to $x$ from $0$ to $\infty$. We have
\begin{align}
	&\frac{1}{\Gamma(\rho+1)}\int_{0}^{\infty}\sideset{}{'}\sum_{\lambda_n \leq x}{} a_{n}(x-\lambda_n)^\rho A(x)dx=\int_{0}^{\infty}Q_\r(x)A(x)dx\\ & +\int_{0}^{\infty}\bigg(\frac{\omega x^{\delta+\rho}}{Q^\delta}\bigg) \sum_{n=1}^{\infty}\bar{b}_n\frac{1}{2 \pi i}\int_{a-i\infty}^{a+i\infty} \frac{\Gamma(\delta-s) \prod_{i=1}^{r}\Gamma(\alpha_i s+\bar{\beta}_i)}{\Gamma(1+\delta-s+\rho)\prod_{i=1}^{r}\Gamma(\alpha_i(\delta-s)+\beta_i)}\bigg(\frac{\mu_nx}{Q^2}\bigg)^{-s}ds A(x)dx\\
	&=K_1(y)+K_2(y) \quad \text{(say).}\label{eq:riesztoauxmod1}
\end{align}
 Using the theory of the multiple Mellin integral \cite[p.53]{TitchmarshFourier} it can be verified that
	\begin{equation}
	A(x)=y^{\r+1}\int_{0}^{\infty}f_{\a_r, \b_r}(u_r)\frac{du_r}{u_r ^{\r+2}}  \cdots \int_{0}^{\infty}f_{\a_1, \b_1}(u_1)f_{\a_0, \b_0} \left(\frac{xy}{u_1u_2 \cdots u_r}\right) \frac{du_1}{u_1 ^{\r+2}}, \label{eq:expansionofA(x)}
\end{equation} 
where $f_{\a_0, \b_0}=\exp(-x)$ and $f_{\a_i, \b_i}= \frac{\exp(-x^{1/\a_i})}{\a_i x^{\b_i/\a_i}}$ for $i=1,\dots, r$. Further using Lemma \ref{lem:exp order of Z_aplha,beta}, we can deduce that
\begin{equation}\label{asymptote for A(x)}
	A(x)\ll \exp{\Big(-c(xy)^{\frac{1}{1+d_F^{'}}}\Big)}.
\end{equation}
Now, the integral on the left-hand side of \eqref{eq:riesztoauxmod1}
 converges absolutely due to the exponential decay of $A(x)$. Hence, using \eqref{eq:expansionofA(x)} and interchanging the integration and summation, we obtain that
\begin{multline}
	\frac{1}{\Gamma(\rho+1)}\int_{0}^{\infty}\sideset{}{'}\sum_{\lambda_n \leq x}{} a_{n}(x-\lambda_n)
	^\rho A(x)dx
\label{lhs}
	\\=\frac{y^{\r+1}}{\G(\r+1)} \sum_{n=1}^{\infty} a_{n} \int_{\lambda_n}^{\infty} \! \int_{0}^{\infty}f_{\a_r, \b_r}(u_r) \frac{du_r}{u_r ^{\r+2}} \cdots \! \int_{0}^{\infty} \! \! f_{\a_1, \b_1}(u_1)f_{\a_0, \b_0} \left(\frac{xy}{u_1u_2 \cdots u_r}\right) \frac{du_1}{u_1 ^{\r+2}}(x-\lambda_n)^\r dx.
	\end{multline}
 Again, due to the absolute convergence of the above, we use Fubini's theorem to interchange the order of integration to see that the above expression equals
\begin{equation}\label{intermidiatestep 1}
	\frac{y^{\r+1}}{\G(\r+1)} \sum_{n=1}^{\infty} a_{n} \! \int_{0}^{\infty}f_{\a_r, \b_r}(u_r) \frac{du_r}{u_r ^{\r+2}} \cdots \int_{0}^{\infty}f_{\a_1, \b_1}(u_1) \left[\int_{\lambda_n}^{\infty} f_{\a_0, \b_0} \left(\frac{xy}{u_1u_2 \cdots u_r}\right) (x-\lambda_n)^\r dx\right] \frac{du_1}{u_1 ^{\r+2}}.
\end{equation}
 As $\r>0$, by the definition of the Gamma function, we have
\begin{equation}
	\int_{\lambda_n}^{\infty} f_{\a_0, \b_0} \left(\frac{xy}{u_1u_2 \cdots u_r}\right) (x-\lambda_n)^\r dx= \G(\r+1) \exp\left(\frac{-\lambda_ny}{u_1\cdots u_r}\right)\left(\frac{u_1 \cdots  u_r}{y}\right)^{\r+1}.
\end{equation}
 Substituting this into \eqref{intermidiatestep 1}, the left-hand side of \eqref{lhs} equals
\begin{multline}
 \sum_{n=1}^{\infty} a_{n} \int_{0}^{\infty}f_{\a_r, \b_r}(u_r) \frac{du_r}{u_r} \cdots \int_{0}^{\infty}f_{\a_1, \b_1}(u_1) \exp\left(\frac{-ny}{u_1\cdots u_r}\right) \frac{du_1}{u_1}.
\end{multline} 
The integral on the right-hand side above is simply $Y_{\bs\alpha,\bs\beta}(ny)$ of Lemma \ref{lem:multiple mellin int}, which yields the left-hand side of our auxiliary modular relation \eqref{aux mod relation}.

We now try to derive the right-hand side of \eqref{aux mod relation}. The integral $K_1(y)$ (in \eqref{eq:riesztoauxmod1}) 
is absolutely convergent as $Q_\r(x)$ is a function of the form $\sum c_\alpha x^{\alpha}$ with finite number of terms and $A(x)$ has exponential decay. Recalling the definition \eqref{resfnQ_rho} of $Q_\r(x)$, we interchange the order of integration to see that $K_1(y)$ is 
\begin{equation} \label{intermidiatestep2}
	\frac{y^{\r+1}}{2 \pi i} \! \int_{C} \! \frac{\phi(s)\Gamma(s) }{\Gamma(s+\rho+1)} ds \! \int_{0}^{\infty} \!\! f_{\a_r, \b_r}(u_r) \frac{du_r}{u_r ^{\r+2}} \cdots \! \int_{0}^{\infty}\!\! f_{\a_1, \b_1}(u_1) \left[\int_{0}^{\infty} \!\! f_{\a_0, \b_0} \left(\frac{xy}{u_1u_2 \cdots u_r}\right) x^{s+\rho} dx\right] \frac{du_1}{u_1 ^{\r+2}}.
\end{equation}
Again, since
\begin{equation}
	\int_{0}^{\infty} f_{\a_0, \b_0} \left(\frac{xy}{u_1u_2 \cdots u_r}\right) x^{s+\rho} dx = \G(s+\r+1) \left(\frac{u_1 \cdots  u_r}{y}\right)^{s+\r+1},
\end{equation}
we obtain 
\begin{equation}
	K_1(y)=\frac{1}{2 \pi i} \int_{C} \phi(s)\Gamma(s)  \int_{0}^{\infty}f_{\a_r, \b_r}(u_r) u_r ^{s-1} du_r \cdots \int_{0}^{\infty}f_{\a_1, \b_1}(u_1) u_1 ^{s-1}  du_1 y^{-s} ds.
\end{equation}
Recalling \eqref{def:gamma fn}, it is evident that $K_1(y)$ is equal to $P_1(y)$, which is the first term on the right-hand side of the needed auxiliary modular relation \eqref{aux mod relation}.

We now turn to $K_2(y)$, which we will denote as $K_2$ for simplicity. By the bound \eqref{asymp of J_rho}, we have
\begin{equation}
	K_2\ll \int_{0}^{\infty} \bigg(\frac{\omega x^{\delta+\rho}}{Q^\delta}\bigg) \sum_{n=1}^{\infty}|b_n| \left(\frac{Q^2}{|\mu_n|x}\right)^{\d+\r} \left(\frac{|\mu_n|x}{Q^2}\right)^{\frac{d_F^{'}\d+(2d_F^{'}-1)\r-\frac{1}{2}}{2d_F^{'}}} A(x)dx.
\end{equation}
Here the series inside the above integral converges absolutely for $\r> (2\s_b-\d)d_F^{'}-\frac{1}{2}$ and thus $K_2$ converges absolutely due to \eqref{asymptote for A(x)}. So we interchange the summation and integral to get
\begin{multline}
	K_2=y^{\r+1}\left(\frac{\omega }{Q^\delta}\right) \sum_{n=1}^{\infty}\bar{b}_n\frac{1}{2 \pi i} \int_{0}^{\infty}\left[\int_{a-i\infty}^{a+i\infty} \frac{\Gamma(\delta-s) \prod_{i=1}^{r}\Gamma(\alpha_i(s)+\bar{\beta_i)}}{\Gamma(1+\delta-s+\rho)\prod_{i=1}^{r}\Gamma(\alpha_i(\delta-s)+\beta_i)}\bigg(\frac{\mu_nx}{Q^2}\bigg)^{-s}ds\right. \times 
	\\
	\left.\int_{0}^{\infty}f_{\a_r, \b_r}(u_r)\frac{du_r}{u_r ^{\r+2}}  \cdots \int_{0}^{\infty}f_{\a_1, \b_1}(u_1)f_{\a_0, \b_0} \left(\frac{xy}{u_1u_2 \cdots u_r}\right) \frac{du_1}{u_1 ^{\r+2}} x^{\delta+\rho} \right]dx.
\end{multline}
Again we use Fubini's theorem to interchange the order of integration and obtain
\begin{multline} \label{intermidiatestep3}
	K_2=y^{\r+1}\bigg(\frac{\omega }{Q^\delta}\bigg) \sum_{n=1}^{\infty}\bar{b}_n\frac{1}{2 \pi i} \int_{a-i\infty}^{a+i\infty} \frac{\Gamma(\delta-s) \prod_{i=1}^{r}\Gamma(\alpha_i(s)+\bar{\beta_i)}}{\Gamma(1+\delta-s+\rho)\prod_{i=1}^{r}\Gamma(\alpha_i(\delta-s)+\beta_i)}\bigg(\frac{\mu_n}{Q^2}\bigg)^{-s}ds \times \\
	\int_{0}^{\infty}f_{\a_r, \b_r}(u_r)\frac{du_r}{u_r ^{\r+2}}  \cdots \int_{0}^{\infty}f_{\a_1, \b_1}(u_1) \left[\int_{0}^{\infty}f_{\a_0, \b_0} \left(\frac{xy}{u_1u_2 \cdots u_r}\right)  x^{\delta+\rho-s} dx \right] \frac{du_1}{u_1 ^{\r+2}}.
\end{multline} 
The term inside the square brackets is $\G(1+\delta-s+\r) \left(\frac{u_1 \cdots  u_r}{y}\right)^{1+\delta-s+\r} $, which gives
\begin{multline}
K_2=\frac{\omega }{(Qy)^\delta}\sum_{n=1}^{\infty}\bar{b}_n\frac{1}{2 \pi i} \int_{a-i\infty}^{a+i\infty} \frac{\Gamma(\delta-s) \prod_{i=1}^{r}\Gamma(\alpha_i(s)+\bar{\beta_i)}}{\prod_{i=1}^{r}\Gamma(\alpha_i(\delta-s)+\beta_i)}\bigg(\frac{\mu_n}{Q^2 y}\bigg)^{-s}ds \times \\
	\int_{0}^{\infty}f_{\a_r, \b_r}(u_r) u_r ^{\delta-s-1} du_r \cdots \int_{0}^{\infty}f_{\a_1, \b_1}(u_1) u_1 ^{\delta-s-1} du_1 .
\end{multline} 
We then use  \eqref{def:gamma fn} to evaluate the integrals above and obtain the second term in the modular relation \eqref{aux mod relation}, which completes the proof of our theorem.

\section{Proof of Theorem \ref{thm:modreln2 implies FE}}

Since the proof is analogous to that of Theorem \ref{thm:mod rel1 implies FE}, we only mention the main steps. Taking the Mellin transform of $Y_{\bs \alpha,\bs\beta}(\lambda_nx)$, multiplying both sides by $Q^s a_{n}$ and summing over $n$, we obtain
	\begin{equation}
	Q^s \phi(s) \G(s) \prod_{i=1}^{r} \Gamma(\alpha_i s+ \beta_i)=Q^s \sum_{n=1}^{\infty} a_n \int_{0}^{\infty} Y_{\bs \alpha,\bs\beta}(\lambda_n x)x^{s-1} dx.
\end{equation}
Using Corollary \ref{cor:exp order of Z_aplha,beta}, one can justify the exchange of summation and integration to get 
\begin{align}
	Q^s \G(s) F(s)&= Q^s\bigg( \int_{0}^{1} \sum_{n=1}^{\infty} a_n Y_{\bs \alpha,\bs\beta}(\lambda_nx)x^{s-1} dx+\int_{1}^{\infty} \sum_{n=1}^{\infty} a_n Y_{\bs \alpha,\bs\beta}(\lambda_nx)x^{s-1} dx\bigg)\\
	&=Q^s\bigg( \int_{0}^{1} P_1(x)x^{s-1} dx+\int_{1}^{\infty}  \big(\Psi(x)x^{-s-1} dx+ \Phi(x)x^{s-1}\big) dx\bigg),\label{aux lhs of fn eqn}
\end{align} on applying the modular relation \eqref{aux mod relation}. Note that here
\begin{equation}
	\Phi(x)=\sum_{n=1}^{\infty} a_n Y_{\bs \alpha,\bs\beta}(\lambda_n x) \hspace{1cm} \text{and} \hspace{1cm} \Psi(x)= \frac{\omega x^\delta}{Q^\delta} \sum_{n=1}^{\infty} \bar{b}_n X_{\bs \alpha,\bs{\bar{\beta}}}\left(\frac{\mu_nx}{Q^2}\right).
\end{equation} We repeat this procedure for 
$X_{\bs \alpha,\bs{\bar{\beta}}}\big(\frac{\mu_nx}{Q^2}\big)$ to obtain 
\begin{equation}
	\omega Q^s  \overline{\psi(\bar{s})} \G(\delta-s)\prod_{i=1}^{r} \Gamma(\alpha_i s+ \bar{\beta_i})
	=\omega Q^{-s} \int_{0}^{\infty}\sum_{n=1}^{\infty} \bar{b}_n  X_{\bs \alpha,\bs{\bar{\beta}}}\bigg(\frac{\mu_nx}{Q^2}\bigg)x^{s-1} dx.
\end{equation} 
Using the modular relation \eqref{aux mod relation}, we obtain
\begin{align}
	\omega Q^s \G(\delta-s) \overline{G(\bar{s})}&= \omega Q^{-s}\bigg(\frac{Q^\delta}{\omega}\int_{0}^{1} \big(\Phi(1/x)-P_1(1/x)\big) x^{-\delta+s-1}+\frac{Q^\delta}{\omega}\int_{1}^{\infty}\sum_{n=1}^{\infty}\Psi(x)x^{-\delta+s-1} dx\bigg)\\
	&=Q^{1-s}\left(-\int_{1}^{\infty}P_1(x)x^{\delta-s-1}dx+\int_{1}^{\infty}\big(\Psi(x)x^{-\delta+s-1}+\Phi(x)x^{\delta-s-1}\big)dx\right). \label{aux rhs of fn eqn}
\end{align} As done after \eqref{rhs of fn eqn}, it can be checked that replacing $s$ by $\d-s$ in \eqref{aux rhs of fn eqn} makes it equal to \eqref{aux lhs of fn eqn}. Thus, we have obtained 
\begin{equation}
	Q^s \G(s) F(s)= 	\omega Q^{\delta-s} \G(s) \overline{G(\delta-\bar{s})},
\end{equation} which gives the functional equation \eqref{fn eqn} since $\G(s)$ does not vanish.

\section{Special cases}

Let us consider the generalized divisor function $\s^{(k)}_z(n)$ defined for $k \in \mathbb{N}$, $z\in \mathbb{C}$ by
\begin{equation}
		\s^{(k)}_z(n):=\sum_{d^k|n}d^z.
\end{equation}
The Dirichlet series associated to $\s^{(k)}_z(n)$ is given by
\begin{equation}
	\sum_{n=1}^{\infty}\frac{\s^{(k)}_z(n)}{n^s}=\z(s)\z(ks-z), \quad \Re(s)>\max\left\{1,\frac{1+\Re(z)}{k}\right\}.
\end{equation}
It can be seen that $\zeta(ks-z)$ satisfies the functional equation 
$$\pi^{-\frac{ks-z}{2}}\Gamma\left(\frac{sk-z}{2}\right)\zeta(ks-z)=\pi^{-\frac{(1-ks+z)}{2}}\Gamma\bigg(\frac{1-ks+z}{2}\bigg)\zeta(1-ks+z).$$
Hence, we can say $\zeta(s)\zeta(ks-z)$ satisfies the functional equation 
\begin{multline} \label{FE for gen.div.fn}
	\pi^{-\frac{s}{2}-\frac{ks-z}{2}}\Gamma\bigg(\frac{s}{2}\bigg) \Gamma\bigg(\frac{ks-z}{2}\bigg)\zeta(s)\zeta(ks-z)\\ =\pi^{-\frac{(1-s)}{2}-\frac{(1-ks+z)}{2}}\Gamma\bigg(\frac{1-s}{2}\bigg) \Gamma\bigg(\frac{1-ks+z}{2}\bigg)\zeta(1-s)\zeta(1-ks+z), 
\end{multline}
which is of the form \eqref{generalfneqn} only if $k=1$ and $\d=z+1$ or  if $z=\frac{k-1}{2}$ and $\d=1$. We discuss our main results for these two cases in the following examples.

\begin{example}
	If we consider the case $k=1$, our function can be denoted by $\s_z(n)$, defined as the sum of the $z^{th}$ powers of the divisors of $n$. The associated Dirichlet series 
$
		\sum_{n=1}^{\infty} \frac{\s_z(n)}{n^s}
$
	satisfies 
	\begin{equation} \label{FE for sigma_z(n)}
		\pi^{-\frac{s}{2}-\frac{s-z}{2}}\Gamma\bigg(\frac{s}{2}\bigg) \Gamma\bigg(\frac{s-z}{2}\bigg)\zeta(s)\zeta(s-z)=\pi^{-\frac{(1-s)}{2}-\frac{(1-s+z)}{2}}\Gamma\bigg(\frac{1-s}{2}\bigg) \Gamma\bigg(\frac{1-s+z}{2}\bigg)\zeta(1-s)\zeta(1-s+z).  
	\end{equation}
	Hence by taking  $a(n)=b(n)= \pi^{\frac{z}{2}}\s_z(n)$ and $\lambda_n=\mu_n=n$, the functional equation \eqref{generalfneqn} in Definition \ref{def:fn eqn of F} is satisfied for $\d=z+1$, $Q=\pi^{-1}$, $\omega$=1, and $r=2$, with $\a_1=\frac{1}{2}$, $\b_1=0$, $\a_2=\frac{1}{2}$, and $\b_2= -\frac{z}{2}$. We consider the various identities that our results are concerned with in the case of this Dirichlet series.
	
	For $x,a>0$, from Theorem \ref{thm:FE implies mod rel1}, we have 
	\begin{equation} \label{modrel1 for sigma_z}
		\sum_{n=1}^{\infty}\pi^{\frac{z}{2}}\s_z(n)Z_{\bs \alpha,\bs\beta}(nx)=P(x)+\left(\frac{\pi}{x}\right)^{z+1}\sum_{n=1}^{\infty}\pi^{\frac{z}{2}}\bar{\s}_z(n) Z_{\bs\alpha,\bs{\beta}}\bigg(\frac{n\pi^2}{x}\bigg),
	\end{equation}
	where $$Z_{\bs \alpha,\bs\beta}(x):=\frac{1}{2 \pi i}\int_{a-i\infty}^{a+i\infty} \Gamma\bigg(\frac{s}{2}\bigg) \Gamma\bigg(\frac{s-z}{2}\bigg) x^{-s}ds,$$
	and 
	\begin{equation}
			P(x)= \frac{1}{2 \pi i}\int_{C}\pi^{\frac{z}{2}}\Gamma\left(\frac{s}{2}\right) \Gamma\left(\frac{s-z}{2}\right)\zeta(s)\zeta(s-z) x^{-s}ds, 
		\end{equation} with $C$ denoting a circle of finite radius, encircling poles of the above integrand.
	Let $x>0$ and $\rho > 2\s-z-\frac{3}{2}$, where $\s=\max\{1,1+\Re(z)\}$. Let $a> \max\{1,1+\Re(z)\}$ be sufficiently large so that all the singularities of $\pi^{\frac{z}{2}}\Gamma\left(\frac{s}{2}\right) \Gamma\left(\frac{s-z}{2}\right)\zeta(s)\zeta(s-z)$ are contained in the strip $z+1-a<\Re(s)<a$. Then, Theorem \ref{thm:FE implies Rieszsum} gives the Riesz-sum identity
	\begin{multline} \label{rieszsum for sigma_z}
		\frac{1}{\Gamma(\rho+1)}\sideset{}{'}\sum_{n\leq  x}{} \pi^{\frac{z}{2}}\s_z(n)(x-n)^\rho\\
		=Q_\rho(x)+ x^{z+1+\rho} \pi^{z+1}\sum_{n=1}^{\infty}\pi^{\frac{z}{2}}\bar{\s}_z(n)\frac{1}{2 \pi i}\int_{a-i\infty}^{a+i\infty} \frac{\Gamma(z+1-s) \Gamma\left(\frac{s}{2}\right) \Gamma\left(\frac{s-z}{2}\right)}{\Gamma(z-s+\rho)\Gamma\left(\frac{1-s}{2}\right) \Gamma\left(\frac{1-s+z}{2}\right)}(nx\pi^2)^{-s}ds	, 
	\end{multline}
	where \begin{equation}
		Q_\rho(x)=\frac{1}{2 \pi i}\int_{C}\frac{\pi^{\frac{z}{2}}\z(s)\z(s-z)\Gamma(s)x^{s+\rho}}{\Gamma(s+\rho+1)} ds,\label{exp:Q_rho}
	\end{equation} with $C=C_a$ denoting a circle of finite radius, lying inside the strip $z+1-a<\Re(s)<a$,  containing all the singularities of $\pi^{\frac{z}{2}}\Gamma\left(\frac{s}{2}\right) \Gamma\left(\frac{s-z}{2}\right)\zeta(s)\zeta(s-z)$, such that all the singularities of $\pi^{\frac{z}{2}}\z(s)\z(s-z)\Gamma(s)x^{s+\rho}$ which lie in this strip are contained inside $C$.
	 For the choices of $a$ and $C$ defined above, the following modular-type relation can be obtained from Theorem \ref{thm:riesz sum implies modreln 2},  
	\begin{multline}\label{modrel2 for sigma_z} 
		\sum_{n=1}^{\infty}\pi^{\frac{z}{2}}\s_z(n)\frac{1}{2 \pi i}\int_{a-i\infty}^{a+i\infty} \G(s)\Gamma\left(\frac{s}{2}\right) \Gamma\left(\frac{s-z}{2}\right) (nx)^{-s}=  
		P_1(x)\\+\left(\frac{\pi}{x}\right)\sum_{n=1}^{\infty}\pi^{\frac{z}{2}}\bar{\s}_z(n) \frac{1}{2 \pi i}\int_{a-i\infty}^{a+i\infty} \G(z+1-s)\Gamma\left(\frac{s}{2}\right) \Gamma\left(\frac{s-z}{2}\right) \left(\frac{n\pi^2}{x}\right)^{-s}ds, 
	\end{multline} 
	where
	\begin{equation}
		P_1(x)= \frac{1}{2 \pi i}\int_{C}  \pi^{\frac{z}{2}}\G(s)\Gamma\left(\frac{s}{2}\right) \Gamma\left(\frac{s-z}{2}\right)\zeta(s)\zeta(s-z)x^{-s}ds.
	\end{equation} 
 By Corollary \ref{cor:equiv of identities}, the relations \eqref{FE for sigma_z(n)}, \eqref{modrel1 for sigma_z}, \eqref{rieszsum for sigma_z}, and \eqref{modrel2 for sigma_z} are equivalent.
\end{example}

\begin{remark}
	If we take $z$ to be an odd integer, the Gamma factors in the functional equation \eqref{FE for sigma_z(n)} can be reduced to a single Gamma factor using Legendre's multiplication formula. In this case, the equivalence of the above identities follows directly from the work of Chandrasekharan and Narasimhan \cite{ChandNarHeck}, which applies in the case of a single Gamma factor.
\end{remark}

\begin{example}	
 Consider $z=\frac{k-1}{2}$. Then the Dirichlet series associated to the arithmetical function $\s^{(k)}_{\frac{k-1}{2}}(n)$ satisfies the functional equation
\begin{multline}
		\pi^{\frac{k-1}{4}-\frac{k+1}{2}s}\Gamma\left(\frac{s}{2}\right) \Gamma\left(\frac{ks}{2}-\frac{k-1}{4}\right)\zeta(s)\zeta\left(ks-\frac{k-1}{2}\right)\\
		=\pi^{-\frac{(k+3)}{4}+\frac{(k+1)}{2}s}\Gamma\left(\frac{1-s}{2}\right) \Gamma\left(\frac{-ks}{2}+\frac{k+1}{4}\right)\zeta(1-s)\zeta\left(-ks+\frac{k+1}{2}\right). 
\end{multline}
It can be seen by taking $a_n=b_n= \pi^{\frac{k-1}{4}}\s^{(k)}_{\frac{k-1}{2}}(n)$ and $\lambda_n=\mu_n=n$, that the functional equation \eqref{generalfneqn} given in Definition \ref{def:fn eqn of F} is satisfied for $\d=1$, $Q=\pi^{-\frac{k+1}{2}}$, $\omega$=1, and $r=2$  with $\a_1=\frac{1}{2}$, $\b_1=0$, $\a_2=\frac{k}{2}$, and $\b_2= -\frac{k-1}{4}$. 

 For $x,a>0$, Theorem \ref{thm:FE implies mod rel1} gives
\begin{equation} \label{modrel1 for sigma^k_z}
	\sum_{n=1}^{\infty}\pi^{\frac{k-1}{4}}\s^{(k)}_{\frac{k-1}{2}}(n)Z_{\bs\alpha,\bs\beta}(nx)=P(x)+\left(\frac{\pi^{\frac{k+1}{2}}}{x}\right)\sum_{n=1}^{\infty}\pi^{\frac{k-1}{4}}\s^{(k)}_{\frac{k-1}{2}}(n) Z_{\bs\alpha,\bs{\bar{\beta}}}\left(\frac{n\pi^{k+1}}{x}\right),
\end{equation}
where $$Z_{\bs \alpha,\bs\beta}(x)=\frac{1}{2 \pi i}\int_{a-i\infty}^{a+i\infty} \Gamma\left(\frac{s}{2}\right) \Gamma\left(\frac{ks}{2}-\frac{k-1}{4}\right) x^{-s}ds,$$
and 
	\begin{equation}
	P(x)= \frac{1}{2 \pi i}\int_{C}\pi^{\frac{k-1}{4}}\Gamma\left(\frac{s}{2}\right) \Gamma\left(\frac{ks}{2}-\frac{k-1}{4}\right)\zeta(s)\zeta\left(ks-\frac{k-1}{2}\right)x^{-s}ds, 
\end{equation} 
with $C$ denoting a circle of finite radius, encircling poles of the above integrand.
Let $x>0$, $\rho >\frac{k}{2}$ and $a>1$ be sufficiently large so that all the singularities of $\pi^{\frac{k-1}{4}}\Gamma\left(\frac{s}{2}\right) \Gamma\left(\frac{ks}{2}-\frac{k-1}{4}\right)\zeta(s)\zeta\left(ks-\frac{k-1}{2}\right)$ are contained in the strip $1-a<\Re(s)<a$. From Theorem \ref{thm:FE implies Rieszsum}, we obtain the Riesz-sum identity
\begin{multline} \label{rieszsum for sigma^k_z}
	\frac{1}{\Gamma(\rho+1)}\sideset{}{'}\sum_{n\leq  x}{} \pi^{\frac{k-1}{4}}\s^{(k)}_{\frac{k-1}{2}}(n)(x-n)^\rho\\
	=Q_\rho(x)+ x^{1+\rho} \pi^{\frac{k+1}{2}}\sum_{n=1}^{\infty}\pi^{\frac{k-1}{4}}\s^{(k)}_{\frac{k-1}{2}}(n)\frac{1}{2 \pi i}\int_{a-i\infty}^{a+i\infty} \frac{\Gamma(1-s) \Gamma\left(\frac{s}{2}\right) \Gamma\left(\frac{ks}{2}-\frac{k-1}{4}\right)}{\Gamma(2-s+\rho)\Gamma\left(\frac{1-s}{2}\right) \Gamma\left(\frac{-ks}{2}+\frac{k+1}{4}\right)}(nx\pi^{k+1})^{-s}ds, 
\end{multline} 
where
 \begin{equation}
	Q_\rho(x)=\frac{1}{2 \pi i}\int_{C}\frac{\pi^{\frac{k-1}{4}}\zeta(s)\zeta\left(ks-\frac{k-1}{2}\right)\Gamma(s)x^{s+\rho}}{\Gamma(s+\rho+1)} ds,\label{exp2:Q_rho}
\end{equation} with $C=C_a$ denoting a circle of finite radius, lying inside the the strip $1-a<\Re(s)<a$,  containing all the singularities of $\pi^{\frac{k-1}{4}}\Gamma\left(\frac{s}{2}\right) \Gamma\left(\frac{ks}{2}-\frac{k-1}{4}\right)\zeta(s)\zeta\left(ks-\frac{k-1}{2}\right)$, such that all the singularities of $\pi^{\frac{k-1}{4}}\zeta(s)\zeta\left(ks-\frac{k-1}{2}\right)\Gamma(s)x^{s+\rho}$ which lie in this strip are contained inside $C$.
For the choices of $a$ and $C$ defined above, Theorem \ref{thm:riesz sum implies modreln 2} gives another modular-type relation
\begin{multline}\label{modrel2 for sigma^k_z} 
	\sum_{n=1}^{\infty}\pi^{\frac{k-1}{4}}\s^{(k)}_{\frac{k-1}{2}}(n)\frac{1}{2 \pi i}\int_{a-i\infty}^{a+i\infty} \G(s)\Gamma\left(\frac{s}{2}\right) \Gamma\left(\frac{ks}{2}-\frac{k-1}{4}\right) (nx)^{-s}=  
	P_1(x)\\+\left(\frac{\pi^{\frac{k+1}{2}}}{x}\right)\sum_{n=1}^{\infty}\pi^{\frac{k-1}{4}}\s^{(k)}_{\frac{k-1}{2}}(n) \frac{1}{2 \pi i}\int_{a-i\infty}^{a+i\infty} \G(1-s)\Gamma\left(\frac{s}{2}\right) \Gamma\left(\frac{ks}{2}-\frac{k-1}{4}\right) \left(\frac{n\pi^{k+1}}{x}\right)^{-s}ds, 
\end{multline} 
where 
\begin{equation}
	P_1(x)= \frac{1}{2 \pi i}\int_{C}  \pi^{\frac{k-1}{4}}\G(s)\Gamma\left(\frac{s}{2}\right) \Gamma\left(\frac{ks}{2}-\frac{k-1}{4}\right)\zeta(s)\zeta\left(ks-\frac{k-1}{2}\right) x^{-s}ds.
\end{equation} 
By Corollary \ref{cor:equiv of identities}, the relations \eqref{FE for sigma_z(n)}, \eqref{modrel1 for sigma^k_z}, \eqref{rieszsum for sigma^k_z}, and \eqref{modrel2 for sigma^k_z} are equivalent.
\end{example}

\begin{example}
	The class of $L$-functions introduced by Selberg \cite{Selbergclass} satisfies the hypothesis of Definition \ref{def:fn eqn of F} with $\d=1$. This class, called the Selberg class, includes well known $L$-functions such as the Dirichlet $L$-functions, Hecke $L$-functions associated with algebraic number fields, the Hecke $L$-functions associated with holomorphic modular forms under appropriate restrictions and normalizations, Artin $L$-functions under the assumption of the Artin conjecture, and the automorphic $L$-functions provided the Ramanujan conjecture holds. Our main result on the equivalence of functional equation, modular relation and Riesz sum identity applies to this broad class of $L$-functions.
\end{example}

\section*{Acknowledgements}
The third author was partially supported by the DST INSPIRE Faculty Award Program  DST/ INSPIRE/Faculty/Batch-13/2018. Some  of this work was carried out when the third author was  visiting  the Harish-Chandra Research Institute, Prayagraj and she is grateful for the kind hospitality. 


\bibliographystyle{abbrv}
\bibliography{rsv}

\end{document}